\documentclass[reqno]{amsart}

\usepackage{amscd,amssymb,amsxtra,amsmath}
\usepackage[pdftex]{hyperref}

\input xy
\xyoption{all}

\usepackage[v2,cmtip]{xy}

\DeclareMathAlphabet{\scr}{OMS}{rsfs}{m}{n}

\DeclareMathAlphabet{\eus}{U}{eus}{m}{n}

\newcommand{\mc}[1]{\mathcal{#1}}

\newcommand{\tens}[1]{\boxtimes_{\mc{#1}}}
\newcommand{\score}[0]{\underline{\quad}}
\newcommand{\bitens}[3]{\mc{#1}\tens{#2}\mc{#3}}
\newcommand{\Buni}[2]{B_{\mc{#1}, \mc{#2}}}

\newcommand{\Efnn}[3]{\underline{Fun}_{#2}({#1}, {#3})}

\newcommand{\Cent}[2]{Z_{#2}({#1})}
\newcommand{\cent}[2]{Z_{\mc{#2}}(\mc{#1})}
\newcommand{\arb}[1]{\ar@<-.15ex>[#1]\ar@<.15ex>[#1]\ar@<-.075ex>[#1]\ar@<.075ex>[#1]\ar[#1]}
\newcommand{\arbdot}[1]{\ar@<-.15ex>@{:}[#1]\ar@<.15ex>@{:}[#1]\ar@<-.075ex>@{:>}[#1]\ar@<.075ex>@{:>}[#1]\ar@{:>}[#1]}
\newcommand{\lb}[0]{\langle}
\newcommand{\rb}[0]{\rangle}
\newcommand{\inthom}[3]{\underline{\textrm{Hom}}_{#2}({#1}, {#3})}
\newcommand{\Hom}[0]{\textrm{Hom}}

\newcommand{\End}[0]{\textrm{End}}
\newcommand{\Rep}[0]{\textrm{Rep}}
\newcommand{\res}[0]{\textrm{res}}
\newcommand{\Forg}[0]{\textrm{Forg}}

\newcommand{\unit}[0]{\textbf{1}}
\newcommand{\Mod}[0]{\textrm{Mod}}

\newtheorem{thm}{Theorem}[section]
\newtheorem{lem}[thm]{Lemma}
\newtheorem{prop}[thm]{Proposition}
\newtheorem{cor}[thm]{Corollary}
\theoremstyle{definition}\newtheorem{defn}[thm]{Definition}
\newtheorem{ex}[thm]{Example}

\newtheorem{remark}[thm]{Remark}

\newtheorem{note}[thm]{Note}
\newtheorem{polyt}[thm]{Polytope}

\begin{document}

\title{monoidal $2$-structure of Bimodule categories}
\author{Justin Greenough}

\email{jrg8@cisunix.unh.edu}
\address{Department of Mathematics and Statistics\\
	University of New Hampshire}
\date{}
\maketitle
\begin{abstract} We define a notion of tensor product of bimodule categories and prove that with this product the 2-category of $\mc{C}$-bimodule categories for fixed tensor $\mc{C}$ is a monoidal 2-category in the sense of Kapranov and Voevodsky (\cite{KV}). We then provide a monoidal-structure preserving 2-equivalence between the 2-category of $\mc{C}$-bimodule categories and $Z(\mc{C})$-module categories (module categories over the center of $\mc{C}$). For a finite group $G$ we show that de-equivariantization is equivalent to the tensor product over $\Rep(G)$. We derive $\Rep(G)$-module fusion rules and show that the group of invertible $\Rep(G)$-module categories is isomorphic to $H^2(G, k^\times)$, extending results in \cite{ENO:homotop}. 
\end{abstract}
\thispagestyle{empty}
\setcounter{tocdepth}{1}
\tableofcontents
\section{Introduction and Main Results}
In this paper we investigate an extension of Deligne's product of abelian categories \cite{CatsTann} to the category $\mc{C}$-bimodule categories. This new product is denoted $\boxtimes_{\mc{C}}$. Here $\mc{C}$ refers to a tensor category over field $k$ which we take, in general, to be of characteristic 0.  This new tensor product reduces to Deligne's product when $\mc{C}=Vec$, the fusion category of finite dimensional $k$-vector spaces.

First steps in defining this extended product involve defining \textit{balanced functors} from the Deligne product of a pair of module categories. This approach mimics classical definitions of tensor product of modules as universal object for balanced morphisms. Tensor product of module categories is then defined in terms of a universal functor factoring balanced functors. In \S \ref{section:2cat} we prove
\begin{thm}\label{thm:main} For any tensor category $\mathcal{C}$, the associated 2-category $\mathcal{B}(\mathcal{C})$ of $\mathcal{C}$-bimodule categories equipped with the tensor product $\boxtimes_{\mathcal{C}}$ becomes a (non-semistrict) monoidal 2-category in the sense of \cite{KV}.
\end{thm}
In \cite{TambaraTensor} constructions similar to these were defined for $k$-linear categories as part of a program to study the representation categories of Hopf algebras and their duals. Balanced functors appeared under the name \textit{bilinear functors}, and the tensor product there is given in terms of generators and relations instead of the universal properties used here. The tensor product was defined and applied extensively by \cite{ENO:homotop} in the study of semisimple module categories over fusion $\mc{C}$.

In order to apply the tensor product of module categories we provide results in \S\ref{tensFunctSect} giving 2-category analogues to classical formulas relating tensor product and hom-functor. In this setting the classical hom functor is replaced by the 2-functor $\underline{Fun}_{\mc{C}}$ giving categories of right exact $\mc{C}$-module functors. 

As an immediate application we prove in \S \ref{DeEquiSect} a result relating de-equivariantization of tensor category $\mc{C}$ to tensor product over $\Rep(G)$, the category of finite dimensional representations of finite group $G$ in $Vec$. Let $A$ be the regular algebra in $\Rep(G)$. Recall (\cite{BraidedI}) that for tensor category $\mc{C}$ \textit{over} $\Rep(G)$ (see Definition \ref{OverDef}) the de-equivariantization $\mc{C}_G$ is defined to be the tensor category of $A$-modules in $\mc{C}$. We prove
\begin{thm}\label{functAlgebra}There is a canonical tensor equivalence $\mc{C}_G\simeq\mc{C}\boxtimes_{\Rep(G)}Vec$ such that the canonical functor $\mc{C}\rightarrow\mc{C}\boxtimes_{\Rep(G)}Vec$ is identified with the canonical (free module) functor $\mc{C}\rightarrow\mc{C}_G$.
\end{thm}

After introducing the notion of center of bimodule category (\S \ref{braidedmodulesection}) we are able to prove a monoidal-structure preserving 2-equivalence between the 2-category of $\mc{C}$-bimodule categories and $Z(\mc{C})$-Mod, module categories over the center $Z(\mc{C})$:
\begin{thm}\label{braidingCenter}There is a canonical monoidal equivalence between $2$-categories $\mc{B}(\mc{C})$ and $Z(\mc{C})$-Mod.
\end{thm}
In \S\ref{sec;ex} we show that, for arbitrary finite group $G$, fusion rules for $\Rep(G)$-module categories over $\boxtimes_{\Rep(G)}$ correspond to products in the twisted Burnside ring over $G$ (see e.g. \cite{OdaYoshida} and \cite{Rose}). As a side effect we show that the group of indecomposable invertible $\Rep(G)$-module categories is isomorphic to $H^2(G, k^{\times})$ thus generalizing results in \cite{ENO:homotop} given for finite abelian groups.


\subsection{Acknowledgements} The author is grateful to the University of New Hampshire for  support during the writing of this paper, and wishes to warmly thank his advisor Dmitri Nikshych for valuable discussion and comments. Without him none of this could have taken place.
%
%
%
%
\section{Preliminaries}\label{preliminariesSection}
Very little in this section is new. Where it seemed necessary to do so we have indicated sources. In most cases what is included here has become standard and so we have omitted references (as general references we suggest \cite{BK}, \cite{K:QG}).
\subsection{Braiding, module categories.}
In this paper all categories are assumed to be abelian and $k$-linear, have finite-dimensional hom spaces, and all functors are assumed to be additive. Even though most of what we do here is valid over fields of positive characteristic we assume at the outset that $k$ is of characteristic 0. All tensor categories are rigid and so they are finite tensor categories in the sense of \cite{EO:FTC}.
\begin{defn}\label{braidingdef} A tensor category $\mc{C}$ is said to be \textit{braided} if it is equipped with a class of natural isomorphisms 
\begin{equation*}
c_{V, W} :V\otimes W\rightarrow W\otimes V
\end{equation*}
for objects $V, W\in\mc{C}$ satisfying the pair of hexagons which can be found in \cite{K:QG} among many other places.
\end{defn}
When $\mc{C}$ is strict these reduce to commuting triangles
\begin{eqnarray}
c_{U, V\otimes W} &=& (id_V\otimes c_{U, W})(c_{U, V}\otimes id_W)\label{eqn:brd1}\\
c_{U\otimes V, W} &=& (c_{U, W}\otimes id_V)(id_U\otimes c_{V, W}) \label{eqn:brd2}.
\end{eqnarray}
In \S \ref{braidedmodulesection} we show how braiding gives module categories bimodule structure.

In the next two examples $G$ is a finite group.
\begin{ex}\label{Rep(G)MonoidEx} $\Rep(G)$, the category of finite dimensional representations of $G$, is a braided tensor category with the usual tensor product. For 2-cocycle $\mu\in Z^2(G, k^\times)$ the category $\Rep_\mu(G)$ of projective representations of $G$ corresponding to Schur multiplier $\mu$ constitutes a tensor category though is in general not braided. 
\end{ex}
\begin{ex} The category $Vec_G^{\omega}$ of finite dimensional $G$-graded vector spaces twisted by $\omega\in H^3(G, k^\times)$ is a rigid monoidal category. Simple objects are given by $k_g$ ($g^{th}$ component $k$, 0 elsewhere) with unit object $k_1$. Associativity is given by $\omega$ and tensor product is defined by
\begin{equation*}
(V\otimes W)_g = \bigoplus_{hk=g}V_h\otimes W_k
\end{equation*}
and $(V^*)_g = ({}^*V)_g = V_{g^{-1}}$. In general $Vec_G^{\omega}$ is not braided.
\end{ex}
\begin{defn}\label{DefCent} The \textit{center} $Z(\mc{C})$ of a monoidal category $\mc{C}$ is the category having as objects pairs $(X, c)$ where $X\in\mc{C}$ and for every $Y\in\mc{C}$ $c_{Y}:Y\otimes X\rightarrow X\otimes Y$ is a family of natural isomorphisms satisfying the hexagon 
\[
\xymatrix{
&(X\otimes Y)\otimes Z\ar[r]^{c_{XY, Z}}&Z\otimes (X\otimes Y)\ar[dr]^{a^{-1}_{Z, X, Y}}&\\X\otimes (Y\otimes Z)\ar[ur]^{a^{-1}_{X, Y, Z}}\ar[dr]_{id_X\otimes c_{Y, Z}}&&&(Z\otimes X)\otimes Y\\&X\otimes (Z\otimes Y)\ar[r]_{a^{-1}_{X, Z, Y}}&(X\otimes Z)\otimes Y\ar[ur]_{c_{X, Z}\otimes id_Y}&
}\] 
for all $Y, Z\in\mc{C}$. Here $a$ is the associativity constraint for the monoidal structure in $\mc{C}$. A morphism $(X, c)\rightarrow (X', c')$ is a morphism $f\in\Hom_{\mc{C}}(X, X')$ satisfying the equation $c'_Y(f\otimes id_Y)=(id_Y\otimes f)c_Y$ for every $Y\in\mc{C}$.
\end{defn}
The center $Z(\mc{C})$ has the structure of a monoidal category as follows. Define the tensor product $(X, c)\otimes(X', c')=(X\otimes X', \tilde{c})$ where $\tilde{c}$ is defined by the composition
\[\xymatrix{
Y\otimes(X\otimes X')\ar[d]_{\tilde{c}_Y}\ar[r]^{a^{-1}_{Y, X, X'}}&(Y\otimes X)\otimes X'\ar[r]^{c_Y}&(X\otimes Y)\otimes X'\ar[d]^{a_{X, Y, X'}}\\
(X\otimes X')\otimes Y&X\otimes (X'\otimes Y)\ar[l]^{a^{-1}_{X, X', Y}}&X\otimes(Y\otimes X')\ar[l]^{c'_Y}
}\]
If $r$ and $\ell$ are the right and left unit constraints for the monoidal structure in $\mc{C}$ then the unit object for the monoidal structure in $Z(\mc{C})$ is given by $(1, r^{-1}\ell)$ as one may easily check. Suppose now that $\mc{C}$ is rigid and $X\in\mc{C}$ has right dual $X^*$. Then $(X, c)\in Z(\mc{C})$ has right dual $(X^*, \overline{c})$ where $\overline{c}_{Y}:=(c^{-1}_{{}^*Y})^*$ and ${}^*Y$ is the left dual of $Y$. One may also check that $Z(\mc{C})$ is braided by $c_{(X, c)\otimes (X', c')}:=c'_{X}$. 

There is a canonical inclusion of monoidal category $\mc{C}$ into its center given by $X\mapsto (X, c_X)$. It is well known that the center $Z(\mc{C})$ is in some sense ``larger" than $\mc{C}$. This differs from the classical analogue in which a ring contains its center. We generalize the notion of center in \S \ref{braidedmodulesection}.

The next definition is essential for this paper.
\begin{defn} A \textit{left module category} $(M, \mu)$ over tensor category $\mc{C}$ is a category $\mc{M}$ together with a bifunctor $\otimes:\mc{C}\times\mc{M}\rightarrow\mc{M}$ and a family of natural isomorphisms $\mu_{X, Y, M}:(X\otimes Y)\otimes M\rightarrow X\otimes (Y\otimes M)$, $\ell_{M}:1\otimes M\rightarrow M$ for $X, Y\in\mc{C}$ and $M\in\mc{M}$ subject to certain natural coherence axioms (see \cite{O:MC}, for example). 
Similarly one defines the structure of \textit{right} module category on $\mc{M}$. If the structure maps are identity we say $\mc{M}$ is \textit{strict} as a module category over $\mc{C}$.
\end{defn}
\begin{note}
It is possible to prove an extended version of MacLane's strictness theorem for module categories which reduces to the monoidal strictness theorem in the regular module case. The proof given in \cite{JGThesis} mimics the proof of the monoidal strictness theorem found in \cite{JS}.
\end{note}
\begin{ex} 
Let $G$ be a finite group with subgroup $H$. The category $\Rep_\mu(H)$ of projective representations of $H$ (Example \ref{Rep(G)MonoidEx}) constitutes a $\Rep(G)$-module category with module category structure defined by $W\otimes V:=\res(W)\otimes V$ whenever $W\in\Rep(G)$, $V\in\Rep_{\mu}(H)$ and $\res:\Rep(G)\rightarrow\Rep(H)$ is the restriction functor. 
\end{ex}
\begin{defn}\label{moduleFunctorDef} For $\mc{M}, \mc{N}$ left $\mc{C}$-module categories a functor $F:\mc{M}\rightarrow\mc{N}$ is said to be a $\mc{C}$-\textit{module functor} if $F$ comes equipped with a family of natural isomorphisms $f_{X, M}:F(X\otimes M)\rightarrow X\otimes F(M)$ satisfying coherence diagrams (again see \cite{O:MC}). We will write $(F, f)$ when referring to such a functor.
A natural transformation $\tau:F\Rightarrow G$ for bimodule functors $(F, f), (G, g):\mc{M}\rightarrow\mc{N}$ is said to be a \textit{module natural transformation} whenever the diagram
\[
\xymatrix{
F(X\otimes M)\ar[rr]^{\tau_{X\otimes M}}\ar[d]_{f_{X, M}}&&G(X\otimes M)\ar[d]^{g_{X, M}}\\
X\otimes F(M)\ar[rr]_{id_X\otimes \tau_{X\otimes N}}&&X\otimes G(M)
}
\]
commutes for all $X\in\mc{C}$ and $M\in\mc{M}$.
\end{defn}
Denote by $Fun_{\mc{C}}(\mc{M}, \mc{N})$ the category of left $\mc{C}$-module functors having morphisms module natural transformations. It is known that this category is abelian and if is semisimple if both $\mc{M}, \mc{N}$ are semisimple (see \cite{ENO:OFC} for details). We will have occasion to deal with categories of right exact module functors and therefore fix notation now.
\begin{defn}\label{rightExactFunctDef} Functor $F:\mc{A}\rightarrow\mc{B}$ is said to be \textit{right exact} if $F$ takes short exact sequences $0\rightarrow A\rightarrow B\rightarrow C\rightarrow 0$ in $\mc{A}$ to sequences $F(A)\rightarrow F(B)\rightarrow F(C)\rightarrow 0$ exact in $\mc{B}$. Similarly one defines left exact functors.  Denote by $\underline{Fun}(\mc{A}, \mc{B})$ the category of right exact functors $\mc{A}\rightarrow\mc{B}$. If $\mc{A}$, $\mc{B}$ are left $\mc{C}$-module categories $\underline{Fun}_{\mc{C}}(\mc{A}, \mc{B})$ is the category of right exact $\mc{C}$-module functors.
\end{defn}
\subsubsection{Bimodule categories.} In much of this paper we will be concerned with categories for which there are left \textit{and} right module structures which interact in a consistent and predictable way. In what follows $\boxtimes$ denotes the product of abelian categories introduced in \cite{CatsTann}.
\begin{defn}\label{defn:bimd}
$\mc{M}$ is a $(\mc{C}, \mc{D})$\textit{-bimodule category} if $\mc{M}$ is a $\mc{C}\boxtimes\mc{D}^{op}$-module category. If $\mc{M}$ and $\mc{N}$ are $(\mc{C}, \mc{D})$-bimodule categories call $F:\mc{M}\rightarrow \mc{N}$ a $(\mc{C}, \mc{D})$-\textit{bimodule functor} if it is a $\mc{C}\boxtimes\mc{D}^{op}$-module functor.
\end{defn}
\begin{note}[Notation] For $\mc{C}$ and $\mc{D}$ finite tensor categories we can define a new category whose objects are $(\mc{C},\mc{D})$-bimodule categories with morphisms $(\mc{C},\mc{D})$-bimodule functors. Denote this category $\mc{B}(\mc{C}, \mc{D})$. When $\mc{C}=\mc{D}$ this is the category of bimodule categories over $\mc{C}$, which we denote $\mc{B}(\mc{C})$. For $\mc{M}$ and $\mc{N}$ in  $\mc{B}(\mc{C}, \mc{D})$ denote by $Fun_{\mc{C}, \mc{D}}(\mc{M}, \mc{N})$ the category of $(\mc{C}, \mc{D})$-bimodule functors from $\mc{M}$ to $\mc{N}$. 
\end{note}
\begin{prop}\label{rmk;bimd} Let $\mc{C}$, $\mc{D}$ be strict monoidal catgories. Suppose $\mc{M}$ has both left $\mc{C}$-module and right $\mc{D}$-module category structures $\mu^l, \mu^r$ and a natural family of isomorphisms $\gamma_{X, M, Y}:(X\otimes M)\otimes Y \rightarrow X\otimes (M\otimes Y)$ for $X$ in $\mc{C}$, $Y$ in $\mc{D}$ making the pentagons 
\[\!\!\!\!\!\!\!\!\!
\xymatrix{
{((XY)M)Z} \ar[d]_{\mu^l\otimes id} \ar[r]^{\gamma} & {(XY)(MZ)} \ar[dd]^{\mu^l} \\ 
{(X(YM))Z} \ar[d]_{\gamma} \\ 
{X((YM)Z)} \ar[r]_{id\otimes \gamma} & {X(Y(MZ))} }
\;\quad
\xymatrix{
{(XM)(YZ)} \ar[d]_{\mu^r} \ar[r]^{\gamma} & {X(M(YZ))} \ar[dd]^{id\otimes \mu^r} \\ 
{((XM)Y)Z} \ar[d]_{\gamma\otimes id} \\ 
{(X(MY))Z} \ar[r]_{\gamma} & {X((MY)Z)}
}
\;\quad
\xymatrix{
(1M)1\ar[r]^{\gamma_{1, M, 1}}\ar[d]_{\ell_M}&1(M1)\ar[dd]^{r_{M}}\\
M1\ar[d]_{r_M}&\\
M&1M\ar[l]^{\ell_M}}
\]
commute. Then $\mc{M}$ has canonical $(\mc{C}, \mc{D})$-bimodule category structure.
\end{prop}
\begin{proof} Straightforward (\cite{JGThesis} contains details).
\end{proof}
\begin{remark}For bimodule structure $(\mc{M}, \mu)$, $\gamma$ is given by $\gamma_{X, M, Y} = \mu_{X\boxtimes 1, 1\boxtimes Y, M}$ over the inherent left and right module category structures. In this way we get the converse of Proposition \ref{rmk;bimd}: every bimodule structure gives separate left and right module category structures and the special constraints described therein in a predictable way.
\end{remark}
\begin{remark}\label{bimodfunct}
We saw in Proposition \ref{rmk;bimd} that bimodule category structure can be described separately as left and right structures which interact in a predictable fashion. We make an analogous observation for bimodule functors. Let $F:(\mc{M}, \gamma)\rightarrow(\mc{N}, \delta)$ be a functor with left $\mc{C}$-module structure $f^{\ell}$ and right $\mc{D}$-module structure $f^r$, where $(\mc{M}, \gamma)$ and $(\mc{N}, \delta)$ are $(\mc{C}, \mc{D})$-bimodule categories with bimodule consistency isomorphisms $\gamma$, $\delta$ as above. Then $F$ is a $(\mc{C}, \mc{D})$-bimodule functor iff the hexagon
\[
\xymatrix{
F(X\otimes (M\otimes Y)) \ar[d]_{f^{\ell}_{X, M\otimes Y}} & F((X\otimes M)\otimes Y) \ar[r]^{f^r_{X\otimes M, Y}} \ar[l]_{\,\,\,F\gamma_{X, M, Y}} & F(X\otimes M)\otimes Y \ar[d]^{f^{\ell}_{X, M}\otimes Y}\\
X\otimes F(M\otimes Y) \ar[r]_{X\otimes f^r_{M, Y}} &X\otimes (F(M)\otimes Y) & (X\otimes F(M))\otimes Y\ar[l]^{\,\,\,\,\,\delta_{X, F(M), Y}}
 }\]
commutes for all $X$ in $\mc{C}$, $Y$ in $\mc{D}$, $M$ in $\mc{M}$. The proof is straightforward and so we do not include it.
\end{remark}
For right $\mc{C}$-module category $\mc{M}$ having module associativity $\mu$ define $\tilde{\mu}_{X, Y, M} = \mu_{M, {}^*Y, {}^*X}$. Then $\mc{M}^{op}$ has left $\mc{C}$-module category structure given by $(X, M) \mapsto M\otimes {}^*X$ with module associativity $\tilde{\mu}^{-1}$. Similarly, if $\mc{M}$ has left $\mc{C}$-module structure with associativity $\sigma$, then $\mc{M}^{op}$ has right $\mc{C}$-module category structure $(M, Y)\mapsto Y^*\otimes M$  with associativity $\tilde{\sigma}^{-1}$ for $\tilde{\sigma}_{M, X, Y} := \sigma_{Y^*, X^*, M}$. 
Lemma \ref{adjnt equiv} simply describes the bimodule structure in the opposite  category of functors. This is a special case of the following proposition.
\begin{prop}\label{bimodprop}
These actions determine a $(\mc{D}, \mc{C})$-bimodule structure 
\begin{equation*}
(Y\boxtimes X, M)\mapsto X^*\otimes M\otimes {}^*Y
\end{equation*}
on $\mc{M}^{op}$ whenever $\mc{M}$ has $(\mc{C}, \mc{D})$-bimodule structure. If $\gamma$ are the bimodule coherence isomorphisms for the left/right module structures in $\mc{M}$ (see Proposition \ref{rmk;bimd}), then $\tilde{\gamma}_{Y, M, X} =\gamma_{X^*, M, {}^*Y}$ are those for $\mc{M}^{op}$.
\end{prop} 
In the sequel whenever $\mc{M}$ is a bimodule category $\mc{M}^{op}$ will always refer to $\mc{M}$ with the bimodule structure described in Proposition \ref{bimodprop}.
In the following definition assume that module category $\mc{M}$ is semisimple over semisimple $\mc{C}$ with finite number of isomorphism classes of simple objects.
\begin{defn}\label{inthomdef} For $M, N\in\mc{M}$ their \textit{internal hom} $\inthom{M}{}{N}$ is defined to be the object in $\mc{C}$ representing the functor Hom$_{\mc{M}}(\score\otimes M, N):\mc{C}\rightarrow Vec$. That is, for any object $X\in\mc{C}$ we have
\begin{equation*}
\Hom_{\mc{M}}(X\otimes M, N)\simeq \Hom_{\mc{C}}(X, \inthom{M}{}{N})
\end{equation*}
naturally in $Vec$. It follows from Yoneda's Lemma that $\inthom{M}{}{N}$ is well defined up to a unique isomorphism and is a bifunctor.
\end{defn}
\subsubsection{Exact module categories.}
It is desirable to restrict the general study of module categories in order to render questions of classification tractable. In their beautiful paper \cite{EO:ftc} Etingof and Ostrik suggest the class of \textit{exact} module categories as an appropriate restriction intermediary between the semisimple and general (non-semisimple, possibly non-finite) cases. Let $P$ be an object in any abelian category. Recall that an object $P$ is called \textit{projective} if the functor $\Hom(P,-)$ is exact. 
\begin{defn} A module category $\mc{M}$ over $\mc{C}$ is called \textit{exact} if for any projective object $P\in\mc{C}$ and any $M\in\mc{M}$, the object $P\otimes M$ is projective.
\end{defn}
It turns out that exactness is equivalent to exactness of certain functors. We will not require the general formulation here, but formulate the next lemma for exact module categories because exactness ensures adjoints for module functors \cite{EO:FTC}.
\begin{lem}\label{Natadjoint} \label{adjnt equiv} For $\mc{M}, \mc{N}$ exact left $\mc{C}$-module categories the association 
\begin{equation*}
Fun_{\mc{C}}(\mc{M}, \mc{N})\stackrel{ad}{\rightarrow} Fun_{\mc{C}}(\mc{N}, \mc{M})^{op}
\end{equation*}
sending $F$ to its left adjoint is an equivalence of abelian categories. If $\mc{M}, \mc{N}$ are bimodule categories then this equivalence is bimodule.
\end{lem}
\begin{proof} Clear.
\end{proof}
\subsection{2-categories and monoidal 2-categories}\label{section:2cat} Recall that a 2-category is a generalized version of an ordinary category where we have cells of various degrees and rules dictating how cells of different degrees interact. There are two ways to compose 2-cells $\alpha, \beta$: \textit{vertical} composition $\beta\alpha$ and \textit{horizontal} composition $\beta*\alpha$ as described by the diagrams below.
\[
\xymatrix
{
A \ar@/^1.65pc/[r]^{f}_{}="1" \ar@/_1.65pc/[r]_{h}^{}="2" \ar[r]|<<<{g}^{}="3" & B \\
& \ar@{=>}"1";"3"^<<<{\beta} \ar@{=>}"3";"2"^<<<{\alpha}
}
\Rightarrow 
\xymatrix
{
A \ar@/^1pc/[r]^{f}_{}="1" \ar@/_1pc/[r]_{h}^{}="2" & B \\
& \ar@<-.5ex>@{=>}"1";"2"^{\alpha\beta} 
},  \;\quad
\xymatrix
{
A \ar@/^1pc/[r]^{f}_{}="1" \ar@/_1pc/[r]_{h}^{}="2" & B \ar@/^1pc/[r]^{f'}_{}="3" \ar@/_1pc/[r]_{h'}^{}="4" & C\\ & \ar@<-.5ex>@{=>}"1";"2"^{\alpha} \ar@<-.5ex>@{=>}"3";"4"^{\beta}
}
\Rightarrow
\xymatrix
{
A \ar@/^1pc/[r]^{f'f}_{}="1" \ar@/_1pc/[r]_{h'h}^{}="2" & C \\
&
\ar@<-.5ex>@{=>}"1";"2"^{\beta*\alpha} 
}
\]
It is required that $\alpha*\beta = (\beta\bullet h)(f'\bullet\alpha) = (h'\bullet\alpha)(\beta\bullet f)$ where $\bullet$ signifies composition between 1-cells and 2-cells giving 2-cells (see \cite{HigherOpsCats} for a thorough treatment of higher category theory and \cite{BenBicat}, \cite{KellyEnriched} for theory of enriched categories). For fixed monoidal category $\mc{C}$ we have an evident 2-category with 0-cells $\mc{C}$-module categories, 1-cells $\mc{C}$-module functors and 2-cells monoidal natural transformations.
\begin{ex}\label{prop:MC2cat} The category of rings defines a 2-category with 0-cells rings, 1-cells bimodules and 2-cells tensor products. 
\end{ex}
A \textit{monoidal} 2-category is essentially a 2-category equipped with a monoidal structure that acts on pairs of cells of various types. For convenience we reproduce, in part, the definition of monoidal $2$-category as it appears in \cite{KV}.
\begin{defn}\label{defnMon2} Let $\mathcal{A}$ be a strict $2$-category. A \textit{(lax) monoidal structure} on $\mathcal{A}$ consists of the following data:
\begin{itemize}
\item[M1.] An object $1 = 1_{\mathcal{A}}$ called the unit object
\item[M2.] For any two objects $A$, $B$ in $\mathcal{A}$ a new object $A\otimes B$, also denoted $AB$
\item[M3.] For any $1$-morphism $u:A\rightarrow A'$ and any object $B$ a pair of $1$-morphisms $u\otimes B:A\otimes B \rightarrow A'\otimes B$ and $B\otimes u:B\otimes A \rightarrow B\otimes A'$
\item[M4.] For any $2$-morphism 
\[ 
\xy 
(-8,0)*+{A}="4"; 
(8,0)*+{A'}="6"; 
{\ar@/^1.65pc/^{u} "4";"6"}; 
{\ar@/_1.65pc/_{u'} "4";"6"}; 
{\ar@{=>}^<<<{T} (0,3)*{};(0,-3)*{}} ; 
\endxy 
\]
and object $B$ there exist $2$-morphisms
\[ 
\xy 
(-11,0)*+{A\otimes B}="4"; 
(11,0)*+{A'\otimes B}="6"; 
{\ar@/^1.65pc/^{u\otimes B} "4";"6"}; 
{\ar@/_1.65pc/_{u'\otimes B} "4";"6"}; 
{\ar@<-1.45ex>@{=>}^<<<{T\otimes B} (0,3)*{};(0,-3)*{}} ; 
\endxy 
\;\quad
\xy 
(-11,0)*+{B\otimes A}="4"; 
(11,0)*+{B\otimes A'}="6"; 
{\ar@/^1.65pc/^{B\otimes u} "4";"6"}; 
{\ar@/_1.65pc/_{B\otimes u'} "4";"6"}; 
{\ar@<-1.45ex>@{=>}^<<<{B\otimes T} (0,3)*{};(0,-3)*{}} ; 
\endxy 
\]
\item[M5.] For any three objects $A$, $B$, $C$ an isomorphism $a_{A, B, C}:A\otimes (B\otimes C)\rightarrow (A\otimes B)\otimes C$
\item[M6.] For any object $A$ isomorphisms $l_{A}:1\otimes A\rightarrow A$ and $r_{A}:A\otimes 1 \rightarrow A$
\item[M7.]\label{M7} For any two morphisms $u:A\rightarrow A'$, $v:B\rightarrow B'$ a $2$-isomorphism 
\[
\xymatrix{ 
A\otimes B 
\ar[r]^{A\otimes v}="1" 
\ar[d]_{u\otimes B}="2" 
& A\otimes B' 
\ar[d]^{u\otimes B'}\\ 
A'\otimes B 
\ar[r]^{A'\otimes v} 
& A'\otimes B' 
\ar@/^.7pc/@<1.5ex>@{=>}"1";"2"^<<<<{\otimes_{u, v}} 
}\]

\item[M8.]\label{M8} For any pair of composable morphisms $A\stackrel{u}{\rightarrow}A'\stackrel{u'}{\rightarrow}A''$ and object $B$ $2$-isomorphisms
\[
\xymatrix{
 A\otimes B \ar[d]_{u\otimes B} \ar[rr]^{(u'u)\otimes B}
           & & A''\otimes B \\ 
 A'\otimes B \ar@<-.85ex>[urr]_{u'\otimes B} \ar@{=>}@<.3ex>[ur]_<<<<<<<<<<<<<<{\otimes_{u, u', B}}
}\;\quad
\xymatrix{
 B\otimes A \ar[d]_{B\otimes u} \ar[rr]^{B\otimes (u'u)}
           & & B\otimes A'' \\ 
 B\otimes A' \ar@<-.85ex>[urr]_{B\otimes u'} \ar@{=>}@<.3ex>[ur]_<<<<<<<<<<<<<<{\otimes_{B, u, u'}}
}
\] 
\item[M9.] For any four objects $A, B, C, D$ a $2$-morphism 
\[
\xymatrix{
A\otimes (B\otimes (C\otimes D)) \ar[d]_{a_{A, B, C\otimes D}\otimes D} \ar[r]^{A\otimes a_{B, C, D}} & A\otimes ((B\otimes C)\otimes D) \ar[dd]^{a_{A, B\otimes C, D}}="1" \\ 
(A\otimes B)\otimes (C\otimes D)\quad\quad \ar[d]_{a_{A\otimes B, C, D}} \\ 
((A\otimes B)\otimes C)\otimes D) & (A\otimes (B\otimes C))\otimes D \ar[l]^{a_{A, B, C}\otimes D}="2" 
\ar@/^-1.1pc/@<-1.5ex>@{=>}"1";"2"_<<<<<<{a_{A, B, C, D}\quad} 
}
\]
\item[M10.] For any morphism $u:A\rightarrow A', v:B\rightarrow B', w:C\rightarrow C'$ $2$-isomorphisms
\[
\xymatrix{ 
A\otimes (B\otimes C) 
\ar[r]^{a_{A, B, C}}="1" 
\ar[d]_{u\otimes (B\otimes C)}="2" 
& (A\otimes B)\otimes C 
\ar[d]^{(u\otimes B)\otimes C}\\ 
A'\otimes (B\otimes C) 
\ar[r]_{a_{A', B, C}} 
& (A'\otimes B)\otimes C 
\ar@/^.9pc/@<1.5ex>@{=>}"1";"2"^<<<<{a_{u, B, C}} 
}\;\quad 
\xymatrix{ 
A\otimes (B\otimes C) 
\ar[r]^{a_{A, B, C}}="1" 
\ar[d]_{A\otimes (v\otimes C)}="2" 
& (A\otimes B)\otimes C 
\ar[d]^{(A\otimes v)\otimes C}\\ 
A\otimes (B'\otimes C) 
\ar[r]_{a_{A, B', C}} 
& (A\otimes B')\otimes C 
\ar@/^.9pc/@<1.5ex>@{=>}"1";"2"^<<<<{a_{A, v, C}} 
}
\] 
\[
\xymatrix{ 
A\otimes (B\otimes C) 
\ar[r]^{a_{A, B, C}}="1" 
\ar[d]_{A\otimes (B\otimes w)}="2" 
& (A\otimes B)\otimes C 
\ar[d]^{(A\otimes B)\otimes w}\\ 
A\otimes (B\otimes C') 
\ar[r]_{a_{A, B, C'}} 
& (A\otimes B)\otimes C' 
\ar@/^.9pc/@<1.5ex>@{=>}"1";"2"^<<<<{a_{A, B, w}} 
}
\]
\item[M11.] For any two objects $A, B$ $2$-isomorphism
\[\xymatrix{
 A\otimes (B\otimes 1) \ar[d]_{a_{A, B, 1}} \ar[rr]^{A\otimes r_B}
           & & A\otimes B \\ 
 (A\otimes B)\otimes 1 \ar[urr]_{r_{A\otimes B}} \ar@{=>}@<.3ex>[ur]_<<<<<<<<<<<<<<{\rho_{A, B}}
}\;\quad 
\xymatrix{
 1\otimes (A\otimes B) \ar[d]_{a_{1, A, B}} \ar[rr]^{l_{A\otimes B}}
           & & A\otimes B \\ 
 (1\otimes A)\otimes B \ar[urr]_{l_A\otimes B} \ar@{=>}@<.3ex>[ur]_<<<<<<<<<<<<<<<<{\lambda_{A, B}}
}
\] 
\[
\xymatrix{
 A\otimes (1\otimes B) \ar[d]_{a_{A, 1, B}} \ar[rr]^{A\otimes l_B}
           & & A\otimes B \\ 
 (A\otimes 1)\otimes B \ar[urr]_{r_A\otimes B} \ar@{=>}@<.3ex>[ur]_<<<<<<<<<<<<<<{\mu_{A, B}}
}
\]
\item[M12.] For any morphism $u:A\rightarrow A'$ $2$-isomorphisms 
\[
\xymatrix{ 
1\otimes A 
\ar[r]^{1\otimes u}="1" 
\ar[d]_{l_A}="2" 
& 1\otimes A' 
\ar[d]^{l_{A'}}\\ 
A 
\ar[r]^{u} 
& A' 
\ar@/^.7pc/@<1.5ex>@{=>}"1";"2"^<<<<<{l_u} 
}
\quad\quad
\xymatrix{ 
A\otimes 1 
\ar[r]^{u\otimes 1}="1" 
\ar[d]_{r_A}="2" 
& A'\otimes 1 
\ar[d]^{r_{A'}}\\ 
A 
\ar[r]^{r} 
& A' 
\ar@/^.7pc/@<1.5ex>@{=>}"1";"2"^<<<<<{r_u} 
}
\]
\item[M13.] A $2$-isomorphism $\epsilon:r_1\Rightarrow l_1$.
\end{itemize}
\end{defn}
These data are further required to satisfy a series of axioms given in the form of \textit{commutative polytopes} listed by Kapranov and Voevodsky. As well as describing the sort of naturality we should expect (extending that appearing in the definition of 2-cells for categories of functors) these polytopes provide constraints on the various cells at different levels and dictates how they are to inteact. For the sake of brevity we do not list them here but will refer to the diagrams in the original paper when needed. In \cite{KV} these polytopes are indicated using hieroglyphic notation. The Stasheff polytope, for example, (which they signify by $(\bullet\otimes\bullet\otimes\bullet\otimes\bullet\otimes\bullet)$, pg. 217) describes how associativity 2-cells and their related morphisms on pentuples of 0-cells interact. In the sequel we will adapt their hieroglyphic notation without explanation. 

We digress briefly to explain what is meant by ``commuting polytope." This notion will be needed for the proof of Theorem \ref{thm:main} Our discussion is taken from \textit{loc. cit.}.
In a strict 2-category $\mc{A}$ algebraic expressions may take the form of 2-dimensional cells subdivided into smaller cells indicating the way in which the larger 2-cells are to be composed. This procedure is referred to as \textit{pasting}. Consider the diagram below left.
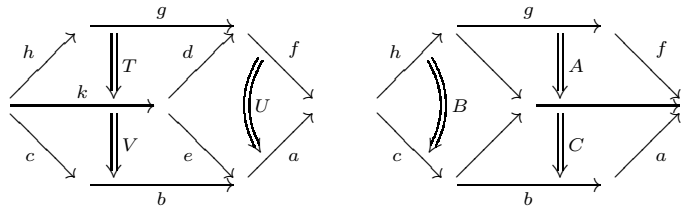
\[
\xymatrix{
& \ar[rr]^<<<<{}="1"^g & & \ar[dr]_<<<<{}="5"^f &\\
\ar[ur]^h \ar[rr]^>>>>>>{}="2"_>>>>>>{}="3"^k \ar[dr]_c & &\ar[ur]^d \ar[dr]_e & &\\
& \ar[rr]^<<<<{}="4"_b & &\ar[ur]^<<<<{}="6"_a &\\
\ar@{=>}"1";"2"^{T} 
\ar@{=>}"3";"4"^{V} 
\ar@/_.6pc/@{=>}"5";"6"^{U}
}\;\quad 
\xymatrix{
&\ar[rr]^>>>>>>{}="1"^g\ar[dr] & & \ar[dr]^f &\\
\ar[ur]_>>>>{}="5"^h \ar[dr]^>>>>{}="6"_c & &  \ar[rr]^<<<<{}="2"_<<<<{}="3" & &\\
& \ar[rr]^>>>>>>{}="4"_b  \ar[ur]& &\ar[ur]_a &\\
\ar@{=>}"1";"2"^{A} 
\ar@{=>}"3";"4"^{C} 
\ar@/^.6pc/@{=>}"5";"6"^{B}
}\]
Edges are 1-cells and faces (double arrows) are 2-cells in $\mc{A}$; $T:gh\Rightarrow dk$, $V:ek \Rightarrow bc$, $U:fd\Rightarrow ae$.  The diagram represents a 2-cell $fgh\Rightarrow abc$ in $\mc{A}$ as follows. It is possible to compose 1-cell $F$ and 2-cell $\alpha$ obtaining new 2-cells $F*\alpha, \alpha * F$ whenever these compositions make sense. If $\alpha:G\Rightarrow H$, these are new 2-cells $FG\Rightarrow FH$ and $GF\Rightarrow HF$, respectively. Pasting of diagram above left represents the composition
\begin{equation*}
fgh \stackrel{f*T}{\Longrightarrow}fdk\stackrel{U*k}{\Longrightarrow}aek\stackrel{a*V}{\Longrightarrow}abc.
\end{equation*}
For 2-composition abbreviated by juxtaposition the pasting is then $(a*V)(U*k)(f*T)$. In case the same external diagram is subdivided in different ways a new 3-dimensional polytope may be formed by gluing along the common edges. Thus the two 2-dimension diagrams can be combined along the edges $fgh$ and $abc$ to form the new 3-dimensional polytope
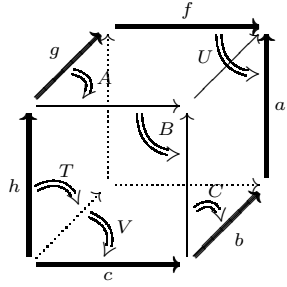
\[
\xymatrix{
& \arb{rr}_>>>>>>{}="11"^f & &\\
\arb{ur}_<<<<<<{}="1"^g \ar[rr]^<<<<<<{}="2"_>>>>>>{}="9" & &\ar[ur] &\\
& \ar@{.>}[uu]|{\,\,\,}\ar@{.>}[rr]|{\,\,} & &\arb{uu}^>>>>>>{}="12" _a\\
\ar@{.>}[ur]_<<<<<<<<<{}="3"_<<<<<<<<<{}="5"\arb{uu}_<<<<<<<<<{}="4" ^h\arb{rr}^<<<<<<<<<{}="6"_c &&\ar[uu]_<<<<<<{}="7"^>>>>>>{}="10"\arb{ur}^<<<<<<{}="8"_b &
\ar@/^.6pc/@{=>}"1";"2"^{A}
\ar@/^.6pc/@{=>}"4";"3"^{T}
\ar@/^.6pc/@{=>}"5";"6"^{V}
\ar@/^.6pc/@{=>}"7";"8"^{C}
\ar@/_.6pc/@{=>}"9";"10"^{B}
\ar@/_.6pc/@{=>}"11";"12"_<<<{U}
}\]
We have labeled only those edges common to the two original figures. As an aid to deciphering polytope commutativity we will denote the boundary with bold arrows as above. To say that the polytope commutes is to say that the results of the pastings of the two sections of its boundary agree. In such a case we say that the pair of diagrams composing the figure are equal: the 2-cells they denote in $\mc{A}$ coincide. 
\section{Balanced Functors and Tensor Products}\label{sec:baltens}
In the remaining sections of this paper we will describe data giving the 2-category of $\mc{C}$-bimodule categories for a fixed tensor $\mc{C}$ the structure of a monoidal 2-category.
\subsection{Preliminary definitions and first properties.} In what follows all module categories are taken over finite tensor categories. Recall definition of tensor functor (Definition \ref{moduleFunctorDef}) and of right exactness (Definition \ref{rightExactFunctDef}).
\begin{defn}\label{DEF;bal}Suppose $(\mc{M}, \mu)$ right, $(\mc{N}, \eta)$ left $\mc{C}$-module categories. A functor $F:\mc{M}\boxtimes \mc{N} \rightarrow \mc{A}$ is said to be \textit{$\mc{C}$-balanced} if there are natural isomorphisms $b_{M, X, Y}:F((M\otimes X)\boxtimes N)\simeq F(M\boxtimes (X\otimes N))$ satisfying the pentagon 
\[
\xymatrix{
F((M\otimes(X\otimes Y))\boxtimes N) \ar[rr]^{b_{M, X\otimes Y, N}} \ar[d]^{\mu_{M, X, Y}}& & F(M\boxtimes((X\otimes Y)\otimes N)) \ar[d]^{\eta_{X, Y, N}}\\
F(((M\otimes X)\otimes Y)\boxtimes N) \ar[dr]_{b_{M\otimes X, Y, N}} & & F(M\boxtimes(X\otimes (Y\otimes N))) \\
& F((M\otimes X)\boxtimes (Y\otimes N)) \ar[ur]_{b_{M, X, Y\otimes N}} }
\]
whenever $X$, $Y$ are objects of $\mc{C}$ and $M\in\mc{M}$.
\end{defn}
\begin{remark} The above occurred in \cite{TambaraTensor} as the definition of ``$k$-bilinear functor" on module categories over $k$-linear tensor categories. The relative tensor product is studied and applied in \cite{ENO:homotop} where many properties are derived in the case that module categories in question are semisimple. 
\end{remark}
Of course Definition \ref{DEF;bal} can be extended to functors from the Deligne product of more than two categories.

\begin{defn}\label{DEF;multbal} Let $F:\mc{M}_1\boxtimes\mc{M}_2\boxtimes\cdots \boxtimes\mc{M}_n\rightarrow \mc{N}$ be a functor of abelian categories and suppose that, for some $i$, $1\leq i \leq n-1$, $\mc{M}_i$ is a right $\mc{C}$-module category and $\mc{M}_{i+1}$ a left $\mc{C}$-module category. Then $F$ is said to be \textit{balanced in the $i^{th}$ position} if there are natural isomorphisms $b^i_{X, M_1,...,M_n}:F(M_1\boxtimes \cdots \boxtimes (M_i\otimes X)\boxtimes M_{i +1}\boxtimes \cdots \boxtimes M_n)\simeq F(M_1\boxtimes \cdots \boxtimes M_i\boxtimes (X\otimes M_{i+1})\boxtimes \cdots\boxtimes M_n)$ whenever $M_i$ are in $\mc{M}_i$ and $X$ is in $\mc{C}$. The $b^i$ are required to satisfy a diagram analogous to that described in Definition \ref{DEF;bal}.
\end{defn}

One may also define \textit{multibalanced} functors $F$ balanced at multiple positions simultaneously.  We will need, and so define, only the simplest nontrivial case.
\begin{defn}\label{def;combal} Let $\mc{M}_1$ be right $\mc{C}$-module, $\mc{M}_2$ $(\mc{C}, \mc{D})$-bimodule, and $\mc{M}_3$ a left $\mc{D}$-module category. The functor $F:\mc{M}_1\boxtimes\mc{M}_2\boxtimes\mc{M}_3\rightarrow
\mc{N}$ is said to be \textit{completely balanced} (or $2$-\textit{balanced}) if for $X\in\mc{C}$, $Y\in\mc{D}$, $N\in\mc{M}_2$, $M\in\mc{M}_1$ and $P\in\mc{M}_3$ there are natural  isomorphisms 
\begin{eqnarray*}
&&b^1_{M, X, N, P}:F((M\otimes X)\boxtimes N \boxtimes P)\simeq F(M\boxtimes (X\otimes N)\boxtimes P)\\
&&b^2_{M, N, Y, P}:F(M\boxtimes (N\otimes Y)\boxtimes P)\simeq F(M\boxtimes N\boxtimes (Y\otimes P))
\end{eqnarray*} 
satisfying the balancing diagrams in Definition \ref{DEF;bal} and the consistency pentagon
\[
\xymatrix{
F((M\otimes X)\boxtimes (N\otimes Y) \boxtimes P) \ar[rr]_{\quad b^2_{M \otimes X, N, Y, P}} \ar[d]|{b^1_{M, X, N\otimes Y, P}}& & F((M\otimes X)\boxtimes N\boxtimes (Y\otimes P)) \ar[dd]^{b^1_{M, X, N, Y\otimes P}}\\
F(M\boxtimes (X\otimes (N\otimes Y))\boxtimes P) \ar[d]_{\gamma^{-1}_{X, N, Y}}&&\\ 
F(M\boxtimes ((X\otimes N)\otimes Y)\boxtimes P) \ar[rr]_{\quad b^2_{M, X, Y\otimes N}}&&F(M\boxtimes (X\otimes N)\boxtimes (Y\otimes P)) }
\]
Here $\gamma$ is the family of natural isomorphisms associated to the bimodule structure in $\mc{M}_2$ (see Remark \ref{rmk;bimd}). Whenever $F$ from $\mc{M}_1\boxtimes\mc{M}_2\boxtimes\cdots \boxtimes\mc{M}_n$ is balanced in ``all" positions call $F$ \textit{$(n-1)$-balanced} or \textit{completely balanced}. In this case the consistency axioms take the form of commuting polytopes. For example the consistency axiom for $4$-balanced functors is equivalent to the commutativity of a polytope having eight faces (four pentagons and four squares) which reduces to a cube on elision of $\gamma$-labeled edges. With this labeling scheme the $1$-balanced functors are the original ones given in Definition \ref{DEF;bal}.
\end{defn}
\begin{defn}\label{DEF;tensor}The \textit{tensor product} of right $\mc{C}$-module category $\mc{M}$ and left $\mc{C}$-module category $\mc{N}$ consists of an abelian category $\mc{M}\boxtimes_{\mc{C}}\mc{N}$ and a right exact $\mc{C}$-balanced functor $B_{\mc{M}, \mc{N}}: \mc{M}\boxtimes \mc{N}\rightarrow \mc{M}\boxtimes_{\mc{C}}\mc{N}$ universal for \textit{right exact} $\mc{C}$-balanced functors from $\mc{M}\boxtimes \mc{N}$. 
\end{defn}
\begin{remark}\label{uniRemark} Universality here means that for any right exact $\mc{C}$-balanced functor $F:\mc{M}\boxtimes \mc{N}\rightarrow \mc{A}$ there exists a unique right exact functor $\overline{F}$ such that the diagram on the left commutes. 
\[
\xymatrix{
 \mc{M}\boxtimes \mc{N} \ar[d]_{B_{\mc{M}, \mc{N}}} \ar[r]^{F}
            & \mc{A} \\ 
 \mc{M}\boxtimes_{\mc{C}} \mc{N} \ar[ur]_{\overline{F}} }
\;\qquad\qquad
\xymatrix{
 \mc{M}\boxtimes \mc{N} \ar[d]_{B_{\mc{M}, \mc{N}}} \ar[drr]^>>>>{F} \ar[rr]^{U}
            && \mc{U} \ar[d]^{F'} \ar@{.>}[dll]_<<<<<<{\alpha} \\ 
 \mc{M}\boxtimes_{\mc{C}} \mc{N} \ar[rr]_{\overline{F}} && \mc{A}
 }
\]
The category $\mc{M}\boxtimes_{\mc{C}} \mc{N}$ and the functor $B_{\mc{M}, \mc{N}}$ are defined up to a unique equivalence. This means that if $U:\mc{M}\boxtimes\mc{N} \rightarrow \mc{U}$ is a second right exact balanced functor with $F=F'U$ for unique right exact functor $F'$, there is a unique equivalence of abelian categories $\alpha :\mc{U}\rightarrow \mc{M}\boxtimes_{\mc{C}}\mc{N}$ making the diagram on the right commute.
\end{remark}
\begin{remark}  The definition of balanced functor may be easily adapted to bifunctors from $\mc{M}\times\mc{N}$ instead of $\mc{M}\boxtimes\mc{N}$. In this case the definition of tensor product becomes object universal for balanced functors \textit{right exact in both variables} from $\mc{M}\times\mc{N}$. This is the approach taken by Deligne in \cite{CatsTann}. One easily checks that our definition reduces to Deligne's for $\mc{C}=Vec$. This provides some justification for defining the relative tensor product in terms of right-exact functors as opposed to functors of some other sort.
\end{remark}
The following lemma is a straightforward application of tensor product universality from Definition \ref{DEF;tensor}. We list it here for later reference. 
\begin{lem}\label{universelemma}
Let $F, G$ be right exact functors $\mc{M}\boxtimes_{\mc{C}}\mc{N}\rightarrow \mc{A}$ such that $F\Buni{M}{N}=G\Buni{M}{N}$. Then $F=G$.
\end{lem}
\begin{proof}
In the diagram
\[
\xymatrix{
 \mc{M}\boxtimes \mc{N} \ar[d]_{B_{\mc{M}, \mc{N}}} \ar[drr]|>>>>>>{T} \ar[rr]^{\Buni{M}{N}}
            &&\mc{M}\boxtimes_{\mc{C}} \mc{N}\ar[d]^{F} \ar@{.>}[dll]_<<<<<<{\alpha} \\ 
 \mc{M}\boxtimes_{\mc{C}}\mc{N}\ar[rr]_{G} && \mc{A}
 }
\]
for $T=F\Buni{M}{N}=G\Buni{M}{N}$ the unique equivalence $\alpha$ is $id_{\mc{M}\boxtimes_{\mc{C}} \mc{N}}$.
\end{proof}
\begin{defn}\label{def;balfunct} For $ \mc{M}$ a right $\mc{C}$-module category and $ \mc{N}$ a left $\mc{C}$-module category denote by $\underline{Fun}^{bal}(\mc{M}\boxtimes \mc{N}, \mc{A})$ the category of right exact $\mc{C}$-balanced functors. Morphisms are natural transformations $\tau : (F, f)\rightarrow (G, g)$ where $f$ and $g$ are balancing isomorphisms for $F$ and $G$ satisfying, whenever $M\in \mc{M}$ and $N\in\mc{N}$,
\[
\xymatrix{
F((M\otimes X)\boxtimes N) \ar[d]_{f_{M, X, N}} \ar[rr]^{\tau_{M\otimes X, N}}
            &&G((M\otimes X)\boxtimes N) \ar[d]^{g_{M, X, N}}\\ 
F((M\boxtimes (X\otimes N)) \ar[rr]_{\tau_{M, X\otimes N}} && G(M\boxtimes (X\otimes N)) }
\] 
for $X$ in $\mc{C}$. Call morphisms in a category of balanced functors \textit{balanced natural transformations}. Similarly we can define $\underline{Fun}^{bal}_i(\mc{M}_1\boxtimes\dots \boxtimes\mc{M}_n, \mc{A})$ to be the category of right exact functors ``balanced in the $i^{th}$ position" requiring of morphisms a diagram similar to that above. 
\end{defn}
It is not obvious at this point that such a universal category exists. The proof of Proposition 3.8 in \cite{ENO:homotop} shows that, in the semisimple case, $\mc{M}\boxtimes_{\mc{C}}\mc{N}$ is equivalent to the center $Z_{\mc{C}}(\mc{M}\boxtimes\mc{N})$ (see \S \ref{braidedmodulesection1} for definitions and discussion). We include the statement here without proof. Functor $I$ is right adjoint to forgetful functor from the center.
\begin{prop}\label{CentTens} There is a canonical equivalence
\begin{equation*}
\mc{M}\boxtimes_{\mc{C}}\mc{N}\simeq Z_{\mc{C}}(\mc{M}\boxtimes\mc{N})
\end{equation*}
such that $I:\mc{M}\boxtimes\mc{N}\rightarrow Z_{\mc{C}}(\mc{M}\boxtimes\mc{N})$ is identified with the universal balanced functor $B_{\mc{M}, \mc{N}}:\mc{M}\boxtimes\mc{N}\rightarrow \mc{M}\boxtimes_{\mc{C}}\mc{N}$.
\end{prop}
\subsection{Module category theoretic structure of tensor product}
In this section we examine functoriality of $\boxtimes_{\mc{C}}$ and discuss module structure of the tensor product.

For $\mc{M}$ a right $\mc{C}$-module category, $\mc{N}$ a left $\mc{C}$-module category, universality of $\Buni{M}{N}$ implies an equivalence between categories of functors
\begin{equation}\label{eqn;equiv}
\mc{Y}:\underline{Fun}^{bal}(\mc{M}\boxtimes \mc{N}, \mc{A}) \stackrel{\sim}{\rightarrow}\underline{Fun}(\mc{M}\tens{C}\mc{N}, \mc{A})
\end{equation}
sending $F\mapsto\overline{F}$ (here overline is as in Definition \ref{DEF;tensor}). Quasi-inverse $\mc{W}$ sends $G\mapsto G\Buni{M}{N}$ with balancing $G*b$, $b$ the balancing of $B_{\mc{M}, \mc{N}}$. On natural transformations $\tau$, $\mc{W}$ is defined by $\mc{W}(\tau) = \tau*\Buni{M}{N}$ where $*$ is the product of 2-morphism and 1-morphism: components are given by $\mc{W}(\tau)_{M\boxtimes N} = \tau_{{\Buni{M}{N}}(M\boxtimes N)}$. One easily checks that $\mc{Y}\mc{W}=id$ so that $\mc{W}$ is a strict right quasi-inverse for $\mc{Y}$. Let $J:\mc{W}\mc{Y}\rightarrow id$ be any natural isomorphism. Then components of $J$ are balanced isomorphisms $J_{(F, f)}:(F, \overline{F}*b)\rightarrow (F, f)$ where $f$ is balancing for functor $F$. Being balanced means commutativity of the diagram 
\[\xymatrix{
F(M\otimes X\tens{}N)\ar[rr]^{\overline{F}(b_{M, X, N})}\ar[d]_{J_{MX\tens{}N}}&&F(M\tens{}X\otimes N)\ar[d]^{J_{M\tens{}XN}}\\
F(M\otimes X\tens{}N)\ar[rr]_{f_{M, X, N}}&&F(M\tens{}X\otimes N)
}\]
for any $M\in\mc{M}, X\in\mc{C}, N\in\mc{N}$. Hence any balancing structure $f$ on the functor $F$ is conjugate to $\overline{F}*b$ in the sense that 
\begin{equation}\label{balancing conjugation equation}
f_{M, X, N}=J_{M\tens{}XN}\circ\overline{F}(b_{M, X, N})\circ J^{-1}_{MX\tens{}N}.
\end{equation}
\begin{remark}\label{remarkComma}Let $F, G:\bitens{M}{}{N}\rightarrow\mc{A}$ be right exact $\mc{C}$-balanced functors. To understand how $\mc{Y}$ acts on balanced natural transformation $\tau:F\rightarrow G$ recall that to any functor $E:\mc{S}\rightarrow\mc{T}$ we associate the \textit{comma category}, denoted $(E, \mc{T})$, having objects triples $(X, Y, q)\in\mc{S}\times\mc{T}\times\Hom_{\mc{T}}(E(X), Y)$. A morphism $(X, Y, q)\rightarrow(X', Y', q')$ is a pair of morphisms $(h, k)$ with the property that $k\circ q = q'\circ E(h)$. For $E$ right exact and $\mc{S}, \mc{T}$ abelian $(E, \mc{T})$ is abelian (\cite{FGR}). 

Let $\overline{F}$ be the unique right exact functor having $\overline{F}\Buni{M}{N}=F$ and consider the comma category $(\overline{F}, \mc{A})$. Natural balanced transformation $\tau$ determines a functor $S_{\tau}:\bitens{M}{}{N}\rightarrow(\overline{F}, \mc{A})$, $X\mapsto(\Buni{M}{N}(X), G(X), \tau_X)$ and $f\mapsto(F(f), G(f))$. It is evident that $S_\tau$ is right exact and inherits $\mc{C}$-balancing from that in $\Buni{M}{N}$, $G$ and $\tau$. Thus we have a unique functor $\overline{S_\tau}:\bitens{M}{C}{N}\rightarrow(\overline{F}, \mc{A})$ with $\overline{S_\tau}\Buni{M}{N}=S_\tau$. Write $\overline{S_\tau}=(S_1, S_2, \sigma)$. Using Lemma \ref{universelemma} one shows that $S_1 = id_{\bitens{M}{C}{N}}$ and $S_2 = \overline{G}$. Then $\sigma(Y):\overline{F}(Y)\rightarrow\overline{G}(Y)$ for $Y\in\bitens{M}{C}{N}$. This is precisely $\overline{\tau}:\overline{F}\rightarrow\overline{G}$.
\end{remark}
Given right exact right $\mc{C}$-module functor $F:\mc{M}\rightarrow \mc{M}'$ and right exact left $\mc{C}$-module functor $G:\mc{N}\rightarrow \mc{N}'$ note that $B_{\mc{M}', \mc{N}'}(F\boxtimes G):\mc{M}\boxtimes \mc{N}\rightarrow \mc{M}'\tens{C}\mc{N}'$ is $\mc{C}$-balanced. Thus the universality of $B$ implies the existence of a unique right exact functor $F\boxtimes_{\mc{C}}G := \overline{B_{\mc{M}', \mc{N}'}(F\boxtimes G)}$ making the diagram 
\[
\xymatrix{
 \mc{M}\boxtimes \mc{N} \ar[d]_{B_{\mc{M}, \mc{N}}} \ar[r]^{F\boxtimes G}
            & \mc{M}'\boxtimes \mc{N}' \ar[d]^{B_{\mc{M}', \mc{N}'}} \\ 
 \mc{M}\boxtimes_{\mc{C}} \mc{N} \ar[r]_{F\boxtimes_{\mc{C}}G} & \mc{M}'\boxtimes_{\mc{C}} \mc{N}'  }
\]
commute. One uses Lemma \ref{universelemma} to show that
$\tens{C}$ is functorial on $1$-cells: $(F'\tens{C}E')(F\tens{C}E) = F'F\tens{C}E'E$. We leave the relevant diagrams for the reader. Thus the 2-cells in M7. of Definition \ref{M8} are identity. If we define $F\otimes\mc{N}:=F\tens{C}id_{\mc{N}}$ (Definition \ref{FMdef}) then the 2-cells in M8. are identity as well.

\begin{remark}\label{tensorOf2-cells} Next consider how $\tens{C}$ can be applied to pairs of module natural transformations. Apply $B_{\mc{N}, \mc{N}'}$ to the right of the diagram for the Deligne product of $\tau$ and $\sigma$\[ \xy (-11,0)*+{\mc{M}\boxtimes\mc{M}'}="4"; (11,0)*+{\mc{N}\boxtimes\mc{N}'}="6"; (33, 0)*+{\mc{N}\tens{C}\mc{N}'} ="8";{\ar@/^1.65pc/^{F\boxtimes E} "4";"6"}; {\ar@/_1.65pc/_{G\boxtimes H} "4";"6"}; {\ar@<-1.45ex>@{=>}^<<<{\tau\boxtimes\sigma} (0,3)*{};(0,-3)*{}} ; {\ar^{B_{\mc{N}, \mc{N}'}} "6";"8"}\endxy \]giving natural transformation \begin{equation}\label{B acting on nats}(\tau\boxtimes\sigma)':=B_{\mc{N}, \mc{N}'}*(\tau\boxtimes\sigma):B_{\mc{N}, \mc{N}'}(F\boxtimes E)\Rightarrow B_{\mc{N}, \mc{N}'}(G\boxtimes H)\end{equation}having components $B_{\mc{N}, \mc{N}'}*(\tau\boxtimes\sigma)_{A\boxtimes B} = B_{\mc{N}, \mc{N}'}(\tau_A\boxtimes\sigma_B)$. Here $*$ indicates composition between cells of different index (in this case a 1-cell and a 2-cell with the usual 2-category structure in \textbf{Cat}). 

It is easy to see that this is a balanced natural transformation, i.e. a morphism in the category of balanced right exact functors $\underline{Fun}^{bal}(\mc{M}\boxtimes\mc{N}, \mc{M}'\tens{C}\mc{N}')$. Using comma category $(F\tens{C}F', \mc{M}'\tens{C}\mc{N}')$ we get
\begin{equation}
\tau\tens{C}\sigma := \overline{(\tau\boxtimes\sigma)'} : F\tens{C}F'\Rightarrow G\tens{C}G'.
\end{equation}
Note also that $\tens{C}$ is functorial over vertical composition of $2$-cells: $(\tau'\tens{C}\sigma')(\tau\tens{C}\sigma) = \tau'\tau\tens{C}\sigma'\sigma$ whenever the compositions make sense. Though we do not prove it here observe also that $\tens{C}$ preserves horizontal composition $\bullet$ of $2$-cells: 
\begin{equation*}
(\tau'\bullet\tau)\tens{C}(\sigma'\bullet\sigma) = (\tau'\tens{C}\sigma')\bullet(\tau\tens{C}\sigma).
\end{equation*} 
\end{remark}
%
For the following proposition recall that, for left $\mc{C}$-module category $\mc{M}$, the functor $L_X:\mc{M}\rightarrow\mc{M}$ sending $M\mapsto X\otimes M$ for $X\in\mc{C}$ fixed is right exact. This follows from the fact that $\Hom(X^*\otimes N, \score)$ is left exact for any $N\in\mc{M}$.
\begin{prop}\label{lem;tensbimd} Let $\mc{M}$ be a $(\mc{C}, \mc{E})$-bimodule category and $\mc{N}$ a $(\mc{E}, \mc{D})$-bimodule category. Then $\mc{M}\boxtimes_{\mc{E}} \mc{N}$ is a $(\mc{C}, \mc{D})$-bimodule category and $B_{\mc{M}, \mc{N}}$ is a $(\mc{C}, \mc{D})$-bimodule functor.
\end{prop}
\begin{proof} For $X$ in $\mc{C}$ define functor $L_X :\mc{M}\boxtimes \mc{N} \rightarrow \mc{M}\boxtimes \mc{N}$ : $M\boxtimes N\mapsto (X\otimes M)\boxtimes N$. Then there is a unique right exact $\overline{L_X}$ making the diagram on the left commute; bimodule consistency isomorphisms in $\mc{M}$ make $L_X$ balanced. 
\[
\xymatrix{
\mc{M}\boxtimes\mc{N}\ar[d]_{B_{\mc{M}, \mc{N}}} \ar[r]^{B_{\mc{M}, \mc{N}}L_X} & {\mc{M}\boxtimes_{\mc{D}} \mc{N}}  \\ 
\mc{M}\boxtimes_{\mc{D}} \mc{N} \ar[ur]_{\overline{L_X}}    }
\;\quad\quad
\xymatrix{
\mc{M}\boxtimes\mc{N}\ar[d]_{B_{\mc{M}, \mc{N}}} \ar[r]^{B_{\mc{M}, \mc{N}}R_Y} & {\mc{M}\boxtimes_{\mc{D}} \mc{N}}  \\ 
\mc{M}\boxtimes_{\mc{D}} \mc{N} \ar[ur]_{\overline{R_Y}}    }
\]
Similarly, for $Y$ in $\mc{D}$ define endofunctor $R_Y:M\boxtimes N\mapsto M\boxtimes (N\otimes Y)$. Then there is unique right exact $\overline{R_Y}$ making the diagram on the right commute; bimodule consistency isomorphisms in $\mc{N}$ make $R_Y$ balanced. $\overline{L_X}$ and $\overline{R_Y}$ define left/right module category structures on $\mc{M}\tens{E} \mc{N}$. Indeed for $\mu$ the left module associativity in $\mc{M}$ note that $B_{\mc{M}, \mc{N}}(\mu_{X, Y, M}\boxtimes id_N):\overline{L_X}\,\overline{L_Y}B_{\mc{M}, \mc{N}} \simeq\overline{L_{X\otimes Y}}B_{\mc{M}, \mc{N}}$ is an isomorphism in $\underline{Fun}^{bal}(\mc{M}\boxtimes\mc{N},  \mc{M}\tens{E} \mc{N})$ and thus corresponds to an isomorphism $\overline{L_X}\,\overline{L_Y} \simeq \overline{L_{X\otimes Y}}$ in $\underline{End}(\mc{M}\tens{E}\mc{N})$ which therefore satisfies the diagram for left module associativity in  $\mc{M}\tens{E} \mc{N}$. Composing diagonal arrows we obtain the following commuting diagram.
\[
\xymatrix{
\mc{M}\boxtimes\mc{N} \ar[d]_{B_{\mc{M}, \mc{N}}} \ar[r]^{B_{\mc{M}, \mc{N}}L_X} & {\mc{M}\boxtimes_{\mc{D}} \mc{N}}  \ar[r]^{\overline{R_Y}} & \mc{M}\boxtimes_{\mc{D}} \mc{N}\\ 
\mc{M}\boxtimes_{\mc{D}} \mc{N} \ar[ur]^{\overline{L_X}} \ar[urr]_{\,\,\overline{R_Y}\,\overline{L_X}}   }
\]
Note then that 
\begin{equation*}
\overline{L_X}\;\overline{R_Y}B_{\mc{M}, \mc{N}}= \overline{R_Y}\, \overline{L_X}B_{\mc{M}, \mc{N}} 
\end{equation*}
and since $\overline{R_Y}\, \overline{L_X}B_{\mc{M}, \mc{N}}$ is balanced Lemma \ref{universelemma} implies $\overline{R_Y}\, \overline{L_X} =  \overline{L_X}\,\overline{R_Y}$. Suppose $Q\in\mc{M}\tens{E} \mc{N}$. Then $(X\boxtimes Y)\otimes Q := \overline{L_X}\,\overline{R_Y}Q = \overline{R_Y}\,\overline{L_X}Q$ defines $(\mc{C}, \mc{D})$-bimodule category structure on  $\bitens{M}{E}{N}$. Note also that since the bimodule consistency isomorphisms in $\mc{M}\boxtimes \mc{N}$ are trivial the same holds in $\bitens{M}{E}{N}$. As a result $\Buni{M}{N}$ is a $(\mc{C}, \mc{D})$-bimodule functor.
\end{proof}
In the sequel we will use $L_X$ to denote left action of $X\in\mc{C}$ in $\mc{M}\boxtimes\mc{N}$ and for the induced action on $\bitens{M}{C}{N}$. Similarly for $R_X$.
%
%
\begin{remark}\label{rmk:bmdmd} The above construction is equivalent to defining left and right module category structures as follows. For the right module structure
\begin{equation*}
\otimes :(\mc{M}\mc{N})\boxtimes \mc{C} \stackrel{\alpha^1_{\mc{M}, \mc{N},\mc{C}}}{\longrightarrow}\mc{M}(\mc{N}\boxtimes \mc{C})\stackrel{id\otimes}{\longrightarrow}\mc{M}\mc{N}
\end{equation*} 
where $\alpha^1$ is defined in Lemma \ref{prop;asslem} and where tensor product of module categories has been written as juxtaposition. The left action is similarly defined using $\alpha^2$ and left module structure of $\mc{M}$ in second arrow. 
\end{remark}
\begin{prop}\label{lem;unit} Let $\mc{M}$ be a $(\mc{C}, \mc{D})$-bimodule category. Then there are canonical $(\mc{C},\mc{D})$-bimodule equivalences $\mc{M}\tens{D}\mc{D}\simeq \mc{M}\simeq \mc{C}\tens{C}\mc{M}$.  
\end{prop}
\begin{proof} Observing that the $\mc{D}$-module action $\otimes$ in $\mc{M}$ is balanced let $l_{\mc{M}}:\mc{M}\tens{D}\mc{D}\rightarrow \mc{M}$ denote the unique exact functor factoring $\otimes$ through $B_{\mc{M}, \mc{D}}$. Define $U:\mc{M}\rightarrow \mc{M}\boxtimes\mc{D}$ by $M\mapsto M\boxtimes1$ and write $U' = \Buni{M}{D}U$. We wish to show that $l_{\mc{M}}$ and $U'$ are inverses. 

Note first that $l_{\mc{M}}U' = id_{\mc{M}}$. Now define natural isomorphism $\tau : \Buni{M}{D}\Rightarrow U'\otimes$ by $\tau_{M, X} = b^{-1}_{M, X, 1}$ where $b$ is balancing isomorphism for $B_{\mc{M}, \mc{D}}$. As a balanced natural isomorphism $\tau$ corresponds to an isomorphism $\overline{\tau}:\overline{\Buni{M}{D}}=id_{\mc{M}\tens{D}\mc{D}}\Rightarrow \overline{U'\otimes}$ in the category $\underline{\End}(\mc{M}\tens{D}\mc{D})$. Commutativity of the diagram
\[
\xymatrix{
\mc{M}\boxtimes\mc{D} \ar[rr]^{\otimes} \ar[d]_{\Buni{M}{D}} && \mc{M} \ar[d]^{U'}\\
\mc{M}\tens{D}\mc{D} \ar[urr]^{l_{\mc{M}}} \ar[rr]_{\overline{U'\otimes}} && \mc{M}\tens{D}\mc{D}
} \]
implies $U'l_{\mc{M}} = \overline{U'\otimes}$ so that $id_{\mc{M}\tens{D}\mc{D}} \simeq U'\l_{\mc{M}}$ via $\overline{\tau}$. In proving $\mc{C}\tens{C}\mc{M}\simeq\mc{M}$ one lifts the left action of $\mc{C}$ for an equivalence $r_{\mc{M}}:\mc{C}\tens{C}\mc{M}\stackrel{\sim}{\rightarrow}\mc{M}$. Strict associativity of the module action on $\mc{M}$ implies that both $r_{\mc{M}}$ and $l_{\mc{M}}$ are trivially balanced. 
\end{proof}
\begin{cor}\label{Functoriality l, r lemma} Let $(F, f):\mc{M}\rightarrow\mc{N}$ be a morphism in $\mc{B}(\mc{C})$ where $f$ is left $\mc{C}$-module linearity for $F$. Then there is a natural isomorphism $Fr_{\mc{M}}\stackrel{\sim}{\rightarrow}r_{\mc{N}}(id_{\mc{C}}\tens{C}F)$ satisfying a polytope version of the diagram for module functors in Definition \ref{moduleFunctorDef}. A similar result holds for the equivalence $l$. 
\end{cor}
\begin{proof} Consider the diagram
\[ 
\xymatrix{
\mc{C}\boxtimes\mc{M} \ar[rrr]^{id_\mc{C}\boxtimes F} \ar[dd]_{\otimes} \ar[dr]_{\Buni{C}{M}} & & &\mc{C}\boxtimes\mc{N}\ar[dl]^{\Buni{C}{N}}\ar[dd]^{\otimes}\\
&\mc{C}\mc{M} \ar[dl]_{r_{\mc{M}}}\ar[r]^{id_\mc{C}\tens{C}F} & \mc{C}\mc{N}\ar[dr]^{r_{\mc{N}}} &\\
\mc{M}\ar[rrr]_{F}& & &\mc{N}
}\]
The top rectangle is definition of $id_{\mc{C}}\tens{C}F$, right triangle definition of functor $r_{\mc{N}}$, and left triangle definition of $r_{\mc{M}}$. The outer edge commutes up to $f$. We therefore have natural isomorphism $f: Fr_{\mc{M}}B_{\mc{C}, \mc{M}}\rightarrow r_{\mc{N}}(id_\mc{C}\tens{C}F)B_{\mc{C}, \mc{N}}$. Now observe that, using the regular module structure in $\mc{C}$ we have the following isomorphisms.
\begin{eqnarray*}
Fr_{\mc{M}}B_{\mc{C}, \mc{M}}(XY\tens{}M)&=&F((XY)M)\\
&=&F(X(YM))=Fr_{\mc{M}}B_{\mc{C}, \mc{M}}(X\tens{}YM),\\
r_{\mc{N}}(id_\mc{C}\tens{C}F)B_{\mc{C}, \mc{N}}(XY\tens{}M)&=&(XY)F(M)\stackrel{\sim}{\rightarrow} XF(YM)\\
&=&r_{\mc{N}}(id_\mc{C}\tens{C}F)B_{\mc{C}, \mc{N}}(X\tens{}YM).
\end{eqnarray*}
Here $X, Y\in\mc{C}$, $M\in\mc{M}$ and $\sim$ is $id_X\otimes f^{-1}_{Y, M}$. Using the relations required of the module structure $f$ described in Definition \ref{moduleFunctorDef} one sees that the second isomorphism constitutes a $\mc{C}$-balancing for the functor $r_{\mc{N}}(id_\mc{C}\tens{C}F)B_{\mc{C}, \mc{N}}$. Thus both functors are balanced. Using the relations for $f$ from Definition \ref{moduleFunctorDef} a second time shows that $f$ is actually a balanced natural isomorphism $Fr_{\mc{M}}B_{\mc{C}, \mc{M}}\rightarrow r_{\mc{N}}(id_\mc{C}\tens{C}F)B_{\mc{C}, \mc{N}}$. Hence we may descend to a natural isomorphism $r_F:=\overline{f}:Fr_{\mc{M}}\rightarrow r_{\mc{N}}(id_\mc{C}\tens{C}F)$. The associated polytopes are given in Polytope \ref{polytope 1} below. The result for $l$ is similar. 
\end{proof}
Corollary \ref{Functoriality l, r lemma} shows, predictably, that functoriality of $l$, $r$ depends on module linearity of the underlying functors. In particular, if $F$ is a strict module functor $l_F$ and $r_F$ are both identity. As an example note that the associativity is strict as a module functor (this follows from Proposition \ref{ass2-natProp}) and so $r_{a_{\mc{M}, \mc{N}, \mc{P}}}=id$ for the relevant module categories. Similarly for $l$. Thus polytopes of the form $(1\otimes\bullet\otimes\bullet\otimes\bullet)$ (pg. 222 in \cite{KV}) describing interaction between $a$, $l$ and $r$ commute trivially. 
\begin{remark}\label{r_r_M=id remark} $r_{\mc{M}}:\mc{C}\tens{C}\mc{M}\rightarrow\mc{M}$ is itself a strict left $\mc{C}$-module functor as follows. Let $X\in\mc{C}$ and let $L_X$ be left $\mc{C}$-module action in $\mc{C}\tens{}\mc{M}$. Replacing $L_X$ with $id\tens{C}F$ in the diagram given in the proof of Corollary \ref{Functoriality l, r lemma} and chasing around the resulting diagram allows us to write the equation 
\begin{equation*}
L'_Xr_{\mc{M}}B_{\mc{C}, \mc{M}}=r_{\mc{M}}\overline{L_X}B_{\mc{C}, \mc{M}}
\end{equation*}
where $L'_X$ is left $X$-multiplication in $\mc{M}$ and $\overline{L_X}$ the induced left $X$-multiplication in $\mc{C}\tens{C}\mc{M}$. Thus $L'_Xr_{\mc{M}}=r_{\mc{M}}\overline{L_X}$, which is precisely the statement that $r_{\mc{M}}$ is strict  as a $\mc{C}$-module functor. Thus Corollary \ref{Functoriality l, r lemma} implies that $r_{r_{\mc{M}}}=id$ for any $\mc{C}$-module category $\mc{M}$. If $\mc{M}$ is a bimodule category it is evident that $r_{\mc{M}}$ is also a strict right module functor and hence strict as a bimodule functor. 
\end{remark}
\begin{prop}\label{lem:FRlem4} For $(\mc{C}, \mc{D})$-bimodule category $\mc{M}$ and $(\mc{C}, \mc{E})$-bimodule category $\mc{N}$ the category of right exact $\mc{C}$-module functors $\underline{Fun}_{\mathcal{C}}(\mathcal{M}, \mathcal{N})$ has canonical structure of a $(\mc{D}, \mc{E})$-bimodule category.
\end{prop}
\begin{proof}
$((X\boxtimes Y)\otimes F)(M) = F(M\otimes X)\otimes Y$ defines $\mathcal{D}\boxtimes\mathcal{E}^{rev}$-action on $\underline{Fun}_{\mathcal{C}}(\mathcal{M}, \mathcal{N})$. Right exactness of $(X\tens{}Y)\otimes F$ comes from right exactness of $F$ and of module action in $\mc{M}, \mc{N}$. $\mathcal{D}\boxtimes\mathcal{E}^{rev}$ acts on the module part $f$ of $F$ by 
\begin{equation*}
((X\boxtimes Y)\otimes f)_{Z, M} = \gamma^{\mc{N}}_{Z, F(M\otimes X), Y} f_{Z, M\otimes X}F(\gamma^{\mc{M}}_{Z, M, X})
\end{equation*}
The required diagrams commute since they do for $f$.

Next let $\tau : F\Rightarrow G$ be a natural left $\mathcal{C}$-module transformation for right exact left $\mathcal{C}$-module functors $(F, f), (G, g) : \mathcal{M}\rightarrow \mathcal{N}$. Define action of $X\boxtimes Y$ on $\tau$ by $((X\boxtimes Y)\otimes \tau)_M = \tau_{M\otimes X}\otimes id_Y : ((X\boxtimes Y)\otimes F)(M) \rightarrow ((X\boxtimes Y)\otimes G)(M)$. Then $(X\boxtimes Y)\otimes \tau$ is a natural left $\mathcal{C}$-module transformation as can be easily checked.
\end{proof}
\begin{remark}\label{lem;unimod} $\mc{Y}$ in equation (\ref{eqn;equiv}) at the beginning of this section is an equivalence of $(\mc{D}, \mc{F})$-bimodule categories
\begin{equation}
\underline{Fun}^{bal}_{\mc{C}}(\mc{M}\boxtimes\mc{N}, \mc{S})\rightarrow \underline{Fun}_{\mc{C}}(\mc{M}\tens{E}\mc{N}, \mc{S})
\end{equation}
whenever $\mc{M}\in\mc{B}(\mc{C}, \mc{E})$, $\mc{N}\in\mc{B}(\mc{E}, \mc{D})$, $\mc{S}\in\mc{B}(\mc{C}, \mc{F})$. If balanced right exact bimodule functor $u:\mc{M}\boxtimes\mc{N}\rightarrow \mc{U}$ is universal for such functors from $\mc{M}\boxtimes\mc{N}$ then $\bitens{M}{E}{N}\simeq \mc{U}$ as bimodule categories. We leave details to the reader. 
\end{remark}
\subsection{Relative tensor product as category of functors.}\label{tensFunctSect} The purpose of this section is to prove an existence theorem for the relative tensor product by providing a canonical equivalence with a certain category of module functors. Let $\mc{M}, \mc{N}$ be exact right, left module categories over tensor category $\mc{C}$, and define $I:\mc{M}\boxtimes\mc{N}\rightarrow \underline{Fun}_{\mc{C}}(\mc{M}^{op}, \mc{N})$ by
\begin{equation*}\label{homiso}
I:M\tens{}N\mapsto\inthom{-}{\mc{M}}{M}\otimes N
\end{equation*}
where $\underline{\Hom}_{\mc{M}}$ means internal hom for right $\mc{C}$-module structure in $\mc{M}$ (Definition \ref{inthomdef}). Using the formulas satisfied by internal hom for right module category structure we see that images under $I$ are indeed $\mc{C}$-module functors:
\begin{eqnarray*}
I(M\tens{}N)(X\otimes M')&=&\inthom{X\otimes M'}{\mc{M}}{M}\otimes N=\inthom{M'}{\mc{M}}{{}^*X\otimes M}\otimes N\\
&=&X\otimes\inthom{M'}{\mc{M}}{M}\otimes N=X\otimes I(M\tens{}N)(M').
\end{eqnarray*}
Using similar relations one easily shows that $I$ is $\mc{C}$-balanced. Hence $I$ descends to a unique right-exact functor $\overline{I}:\bitens{M}{C}{N}\rightarrow\underline{Fun}_{\mc{C}}(\mc{M}^{op}, \mc{N})$ satisfying $\overline{I}\Buni{M}{N}=I$.

In the opposite direction define $J:\underline{Fun}_{\mc{C}}(\mc{M}^{op}, \mc{N})\rightarrow\bitens{M}{}{N}$ as follows. For $F$ a $\mc{C}$-module functor $\mc{M}^{op}\rightarrow\mc{N}$ let $J(F)$ be the object representing the functor $M\boxtimes N\mapsto\Hom(N, F(M))$, that is $\Hom_{\bitens{M}{}{N}}(M\tens{}N, J(F))=\Hom_{\mc{N}}(N, F(M))$. Now denote by $J':\underline{Fun}_{\mc{C}}(\mc{M}^{op}, \mc{N})\rightarrow\bitens{M}{C}{N}$ the composition $\Buni{M}{N}J$. 
\begin{thm}\label{prop;exist} Let $\mc{C}$ be a rigid monoidal category. For $\mc{M}$ a right $\mc{C}$-module category and $\mc{N}$ a left $\mc{C}$-module category there is a canonical equivalence
\begin{equation*}
\mc{M}\boxtimes_{\mc{C}} \mc{N} \simeq \underline{Fun}_{\mc{C}}(\mc{M}^{op}, \mc{N}).
\end{equation*}
If $\mc{M}, \mc{N}$ are bimodule categories this equivalence is bimodule.
\end{thm}
\begin{proof} In order to prove the theorem we simply show that $\overline{I}$ and $J'$ defined above are quasi-inverses. This will follow easily if we can first show that $I, J$ are quasi-inverses, and so we dedicate a separate lemma to proving this. 
\begin{lem}\label{IJQInverseLemma} $I, J$ are quasi-inverses.
\end{lem}
\begin{proof} Let us first discuss internal homs for the $\mc{C}$-module structure in $\bitens{M}{}{N}$ induced by $X\otimes (M\tens{}N):=(X\otimes M)\tens{}N$. Let $X$ be any simple object in $\mc{C}$. Then one shows, using the relations for internal hom in $\mc{M}$ and $\mc{N}$ separately, that the internal hom in $\bitens{M}{}{N}$ is given by 
\begin{equation}\label{DeligneIntHom}
\inthom{M\tens{}N}{\bitens{M}{}{N}}{S\tens{}T}=\inthom{M}{\mc{M}}{S}\otimes\inthom{N}{\mc{N}}{T}
\end{equation}
where the $\otimes$ is of course that in $\mc{C}$. Using this and the definitions of $I$ and $J$ we have
\begin{eqnarray*}
\Hom_{\bitens{M}{}{N}}(M\tens{}N, JI(S\tens{}T))&=&\Hom_{\mc{N}}(N, \inthom{M}{\mc{M}}{S}\otimes T)\\
&=&\Hom_{\mc{C}}(1, \inthom{N}{\mc{N}}{\inthom{M}{\mc{M}}{S}\otimes T})\\
&=&\Hom_{\mc{C}}(1, \inthom{M\tens{}N}{\bitens{M}{}{N}}{S\tens{}T})\\
&=&\Hom_{\bitens{M}{}{N}}(M\tens{}N, S\tens{}T).
\end{eqnarray*}
The third line is an application of (\ref{DeligneIntHom}). The first and the last line imply that the functor $M\tens{}N\mapsto \Hom_{\mc{N}}(N, \inthom{M}{\mc{M}}{S}\otimes T)$ is represented by both $S\tens{}T$ and $JI(S\tens{}T)$, and these objects must therefore be equal up to a unique isomorphism, hence $JI\simeq id$. 

Next we show that $IJ\simeq id$. Let $F$ be any functor $\mc{M}^{op}\rightarrow\mc{N}$. From the first part of this proof we may write the following equation (up to unique linear isomorphism):
\begin{eqnarray*}
\Hom_{\mc{N}}(N, IJ(F)(M))&=&\Hom_{\bitens{M}{}{N}}(M\tens{}N, JIJ(F))\\
&=&\Hom_{\bitens{M}{}{N}}(M\tens{}N, J(F))=\Hom_{\mc{N}}(N, F(M)).
\end{eqnarray*}
Thus both $IJ(F)(M)$ and $F(M)$ are representing objects for the functor $N\mapsto\Hom_{\bitens{M}{}{N}}(M\tens{}N, J(F))$ for each fixed $M\in\mc{M}$. Thus $IJ(F)(M)=F(M)$ up to a unique isomorphism. The collection of all such isomorphisms gives a natural isomorphism $IJ(F)\simeq F$, and therefore $IJ\simeq id$. This, with the first part of this proof, is equivalent to the statement that $J$ is a quasi-inverse for $I$, proving the lemma.
\end{proof}
Now we are ready to complete the proof of Theorem \ref{prop;exist}. Using the definition of $J'$ and $\overline{I}$ write $J'\overline{I}\Buni{M}{N}=\Buni{M}{N}JI \simeq\Buni{M}{N}$. By uniqueness (Lemma \ref{universelemma}) it therefore follows that $J'\overline{I}\simeq id$. Also $\overline{I}J'=\overline{I}\Buni{M}{N}J = IJ\simeq id$, and we are done.
\end{proof}
\subsubsection{Adjunction with category of functors}
As an immediate corollary to Theorem \ref{prop;exist} and associativity of relative tensor product (equation \ref{partialAss}, given below) we are able to prove a module category theoretic version of a theorem which appears in many connections in the classical module-theory literature . 

\begin{cor}[Frobenius Reciprocity]\label{cor;frbns} Let $\mc{M}$ be a $(\mc{C}, \mc{D})$-bimodule category, $\mc{N}$ a $(\mc{D}, \mc{F})$-module category, and $\mc{A}$ a $(\mc{C}, \mc{F})$-module category. Then there is a canonical equivalence
\begin{equation} \label{eqn:frbns}
\underline{Fun}_{\mc{C}}(\mc{M}\tens{D}\mc{N}, \mc{A})\simeq \underline{Fun}_{\mc{D}}(\mc{N}, \underline{Fun}_{\mc{C}}(\mc{M}, \mc{A}))
\end{equation}
as $(\mc{E}, \mc{F})$-bimodule categories.
\end{cor}
\begin{proof} To see this we will first use Lemma \ref{Natadjoint} to describe the behaviour of the tensor product under $op$. Observe that
\begin{equation*}
(\bitens{M}{D}{N})^{op}\simeq\underline{Fun}_{\mc{D}}(\mc{M}^{op}, \mc{N})^{op}\simeq\underline{Fun}_{\mc{D}}(\mc{N}, \mc{M}^{op})\simeq\mc{N}^{op}\tens{D}\mc{M}^{op}
\end{equation*}
applying Theorem \ref{prop;exist} twice (first and third) and Lemma \ref{Natadjoint} for the second step. Now we may write
\begin{eqnarray*}
\underline{Fun}_{\mc{C}}(\bitens{M}{D}{N}, \mc{A})\simeq (\bitens{M}{D}{N})^{op}\tens{C}\mc{A}&\simeq&(\mc{N}^{op}\tens{D}\mc{M}^{op})\tens{C}\mc{A}\\
&\simeq&\mc{N}^{op}\tens{D}(\mc{M}^{op}\tens{C}\mc{A})\\
&\simeq&\underline{Fun}_{\mc{C}}(\mc{N}, \underline{Fun}_{\mc{D}}(\mc{M}, \mc{A})).
\end{eqnarray*}
\end{proof}
Theorem \ref{cor;frbns} states that functor $\mc{M}\tens{D}-:\mc{B}(\mc{D}, \mc{E})\rightarrow \mc{B}(\mc{C}, \mc{E})$ is left adjoint to functor $\Efnn{\mc{M}}{\mc{C}}{-} :\mc{B}(\mc{C}, \mc{E})\rightarrow \mc{B}(\mc{D}, \mc{E})$.
\section{Associativity and unit constraints for $\mc{B}(\mc{C})$}
\subsection{Tensor product associativity} 
In this section we discuss associativity of tensor product. Let $\mc{C}, \mc{D}, \mc{E}$ be  tensor categories. Let $\mc{A}$ be a right $\mc{C}$-module category, $\mc{M}$ a $\mc{C}$-$\mc{D}$-bimodule category, $\mc{N}$ a $\mc{D}$-$\mc{E}$-bimodule category and $\mc{P}$ a left $\mc{E}$-module category. In an effort to save space we will at times abbreviate tensor product by juxtaposition. 
\begin{lem}\label{prop;asslem} $\mc{A}\boxtimes (\mc{M}\boxtimes_{\mc{D}}\mc{N}) \simeq (\mc{A}\boxtimes \mc{M})\boxtimes_{\mc{D}}\mc{N}$ and $(\bitens{M}{D}{N})\boxtimes \mc{A}\simeq\mc{M}\tens{D} (\mc{N}\boxtimes \mc{A})$ as abelian categories. 
\end{lem}
\begin{proof} Let $F: \mc{A}\boxtimes\mc{M}\boxtimes\mc{N}\rightarrow\mc{S}$ be totally balanced (Definition \ref{def;combal}). For $A$ in $\mc{A}$ define functor $F_A:\mc{M}\boxtimes \mc{N}\rightarrow \mc{S}$ by $M\boxtimes N \mapsto F(A\boxtimes M\boxtimes N)$ on simple tensors and $f\mapsto F(id_A\boxtimes f)$ on morphisms. Note that functors $F_A$ are balanced since $F$ is totally balanced. Thus for any object $A$ there is a unique functor $\overline{F_A}: \mc{M}\tens{D}\mc{N}\rightarrow \mc{S}$ satisfying the diagram below left. The $\overline{F_A}$ allow us to define functor $F':\mc{A}\boxtimes(\mc{M}\tens{D} \mc{N})\rightarrow \mc{S}:A\boxtimes Q\mapsto \overline{F_A}(Q)$ whenever $Q$ is an object of $\mc{M}\tens{D}\mc{N}$ giving the commutative upper right triangle in the diagram on the right.
\[ 
\xymatrix{
\mc{M}\boxtimes \mc{N}\ar[d]_{B_{\mc{M}, \mc{N}}}\ar[dr]^{F_A}&\\
\bitens{M}{D}{N}\ar[r]_<<<<{\overline{F_A}}&\mc{S}
}
\;\qquad\quad
\xymatrix { 
\mc{A}\boxtimes\mc{M}\boxtimes\mc{N} \ar[d]_{B_{\mc{A}\boxtimes\mc{M}, \mc{N}}} \ar[dr]^{F} \ar[r]^{B_{\mc{M}, \mc{N}}} & \mc{A}\boxtimes(\mc{M}\boxtimes_{\mc{D}} \mc{N}) \ar[d]_{F'}\ar@{.>}[dl]\\
(\mc{A}\boxtimes\mc{M})\boxtimes_{\mc{D}}\mc{N} \ar[r]_<<<<<<<<{\overline{F}}& \mc{S}
} \]
Since the functors $B_{\mc{A}\boxtimes\mc{M}, \mc{N}}$, $B_{\mc{M}, \mc{N}}$, $\overline{F}$ and $F'$ are unique by the various universal properties by which they are defined, both $\mc{A}\boxtimes(\mc{M}\boxtimes_{\mc{D}} \mc{N}) $ and $(\mc{A}\boxtimes\mc{M})\boxtimes_{\mc{D}} \mc{N}$ are universal factorizations of $F$ and must therefore be connected by a unique equivalence 
\begin{equation*}
\alpha^2_{\mc{A}, \mc{M}, \mc{N}}: \mc{A}\boxtimes(\bitens{M}{D}{N})\stackrel{\sim}{\rightarrow}(\mc{A}\boxtimes\mc{M})\tens{D}\mc{N}
\end{equation*}
(perforated arrow in diagram). One obtains natural equivalence $\alpha^1_{\mc{M}, \mc{N}, \mc{A}}:(\bitens{M}{D}{N})\boxtimes \mc{A}\stackrel{\sim}{\rightarrow}\mc{M}\tens{D} (\mc{N}\boxtimes \mc{A})$ by giving the same argument ``on the other side," i.e. by first defining $F_N:\mc{A}\boxtimes\mc{M}\rightarrow \mc{S}$ for fixed $N\in\mc{N}$ and proceeding analogously.
\end{proof}
\begin{remark} For bimodule category $\mc{A}$ Remark \ref{lem;unimod} implies that $\alpha^i$ are bimodule equivalences.
\end{remark}
\begin{lem}\label{lem:assbal} For $\alpha^1$ in Lemma \ref{prop;asslem} $(\mc{A}\tens{C}\Buni{M}{N})\alpha^1_{\mc{A}, \mc{M}, \mc{N}}:(\bitens{A}{C}{M})\tens{}\mc{N}\rightarrow \mc{A}\tens{C}(\bitens{M}{D}{N})$ is balanced.
\end{lem}
\begin{proof}
Treat $\mc{M}$ as having right $\mc{C}$-module structure coming from its bimodule structure, and similarly give $\mc{N}$ its left $\mc{C}$-module structure. Recall, as above, we define $R_X:\mc{M}\rightarrow \mc{M}$ and $L_X:\mc{N}\rightarrow\mc{N}$ right and left action of $X\in\mc{C}$ on $\mc{M}, \mc{N}$ respectively. We will use superscripts to keep track of where $\mc{C}$-action is taking place, e.g. $R^{\mc{M}}_Y$ means right action of $X$ in $\mc{M}$. Recall also $X\in\mc{D}$ induces right $\mc{D}$-action $id_{\mc{A}}\tens{C}R_X:\bitens{A}{C}{M}\rightarrow \bitens{A}{C}{M}$ which we denote also by $R_X$. Consider the following diagram:
\[
\!\!\!\!\!\!\!\!
\xymatrix{
\mc{A}\boxtimes\mc{M}\boxtimes\mc{N}\ar[rrr]^{B\boxtimes id}\ar[ddd]_{B\boxtimes id}\ar[dddr]_{R^{\mc{M}}_X}\ar[rrrd]^>>>>>>>>>>>>>{L^{\mc{N}}_X}\ar@/^-1.2pc/[ddrr]_>>>>>>>>{B(R_X\boxtimes id)}_{}="1"\ar@/^1.2pc/[ddrr]^>>>>>>>>{B(id\boxtimes L_X)}^{}="2"&&&\mc{A}\mc{M}\boxtimes\mc{N}\ar[r]^{id\boxtimes L_X}&\mc{A}\mc{M}\boxtimes\mc{N}\ar[dd]^{\alpha^1}\\
&&&\mc{A}\boxtimes\mc{M}\boxtimes\mc{N}\ar[ur]_{B\boxtimes id}\ar[dl]^{id\boxtimes B}\ar[dr]_B&\\
&&\mc{A}\boxtimes\mc{M}\mc{N}\ar[ddrr]^B&&\mc{A}(\mc{M}\boxtimes\mc{N})\ar[dd]^{id\boxtimes_*B}\\
\mc{A}\mc{M}\boxtimes\mc{N}\ar[d]_{R_X\boxtimes id}&\mc{A}\boxtimes\mc{M}\boxtimes\mc{N}\ar[dl]^{B\boxtimes id}\ar[ur]_{id\boxtimes B}\ar[dr]_{B}&&&\\
\mc{A}\mc{M}\boxtimes\mc{N}\ar[rr]_{\alpha^1}&&\mc{A}(\mc{M}\boxtimes\mc{N})\ar[rr]_{id\boxtimes_*B}&&\mc{A}(\mc{M}\mc{N})\\
\ar@{=>}"1";"2"^{\simeq}_{b}
}
\]
Leftmost rectangle is $(\textrm{definition of }R_X)\boxtimes id_{\mc{N}}$, top rectangle is tautologically $B\boxtimes L_X$, upper right and lower left triangles are definition of $\alpha^1$, lower right rectangles definition of $id_{\mc{A}}\tens{C}\Buni{M}{N}$ and $b$ is $id_\mc{A}\boxtimes(\textrm{balancing isomorphism for }B_{\mc{M}, \mc{N}})$. An application of Lemma \ref{universelemma} then gives
\begin{equation*}
(id_{\mc{A}}\tens{C}\Buni{M}{N})\alpha^1_{\mc{A}, \mc{M}, \mc{N}}(R_X\boxtimes id_{\mc{N}}) \stackrel{b}{\simeq} (id_{\mc{A}}\tens{C}\Buni{M}{N})\alpha^1_{\mc{A}, \mc{M}, \mc{N}}(id_{\bitens{A}{C}{M}}\boxtimes L_X)
\end{equation*}
Since $b$ satisfies the balancing axiom (Definition \ref{DEF;bal}) for $\Buni{M}{N}$ it satisfies it here. This is precisely the statement that $(\mc{A}\tens{C}\Buni{M}{N})\alpha^1_{\mc{A}, \mc{M}, \mc{N}}$ is balanced.
\end{proof}
\begin{prop}\label{prop;ass1} If $\mc{A}$ and $\mc{N}$ are bimodules we have $(\mc{A}\tens{C}\mc{M})\tens{D}\mc{N} \simeq \mc{A}\tens{C}(\mc{M}\tens{D}\mc{N})$ as bimodule categories. 
\end{prop}
\begin{proof}
We plan to define the stated equivalence as the image of the functor $(\mc{A}\tens{C}\Buni{M}{N})\alpha^1_{\mc{A}, \mc{M}, \mc{N}}:(\bitens{A}{C}{M})\boxtimes{N}\rightarrow \mc{A}\tens{C}(\bitens{M}{D}{N})$ under $\mc{Y}$ (equation (\ref{eqn;equiv})). Lemma \ref{lem:assbal} implies that indeed $\mc{Y}$ is defined there. With notation as above define $a^1$ and $a^2$ using the universality of $B$ by the following diagrams.
\[
\xymatrix{
(\bitens{A}{C}{M})\boxtimes{N} \ar[r]^{\alpha^1_{\mc{A}, \mc{M},\mc{N}}} \ar[d]_{\Buni{AM}{N}} & \mc{A}\tens{C}(\mc{M}\boxtimes\mc{N}) \ar[d]|{id_\mc{A}\tens{C}\Buni{M}{N}} \\
(\bitens{A}{C}{M})\tens{D}\mc{N} \ar[r]_{a^1_{\mc{A}, \mc{M}, \mc{N}}} & \mc{A}\tens{C}(\bitens{M}{D}{N})}
\;\quad
\xymatrix{
\mc{A}\boxtimes (\bitens{M}{D}{N}) \ar[r]^{\alpha^2_{\mc{A}, \mc{M},\mc{N}}} \ar[d]_{\Buni{A}{MN}} & (\mc{A}\boxtimes\mc{M})\tens{D}\mc{N} \ar[d]|{\Buni{A}{M}\tens{D}id_\mc{N}} \\
\mc{A}\tens{C}(\bitens{M}{D}{N}) \ar[r]_{a^2_{\mc{A}, \mc{M}, \mc{N}}} & (\bitens{A}{C}{M})\tens{D}\mc{N}
}
\]
$\alpha^i$ are defined in Lemma \ref{prop;asslem}.To see that $a^1$ and $a^2$ are quasi-inverses consider the diagram
\[
\xymatrix{
\mc{A}\boxtimes(\mc{M}\mc{N})  \ar[dd]|{\Buni{A}{MN}} \ar[dr]_{\alpha^2} & & \mc{A}\boxtimes\mc{M}\boxtimes\mc{N} \ar[ll]_{id_\mc{A}\boxtimes\Buni{M}{N}} \ar[rr]^{B_{\mc{A}, \mc{M}\boxtimes\mc{N}}} \ar[dl]|{B_{\mc{A}\boxtimes\mc{M}, \mc{N}}} \ar[dr]|{B_{\mc{A}, \mc{M}}\boxtimes id_\mc{N}} & &\mc{A}(\mc{M}\boxtimes\mc{N}) \ar[dd]|{id_\mc{A}\tens{C}\Buni{M}{N}}\\
& (\mc{A}\boxtimes\mc{M})\mc{N} \ar[dr]|{\Buni{A}{M}\tens{D} id_\mc{N}} & &(\mc{A}\mc{M})\boxtimes\mc{N}\ar[ur]_{\alpha^1} \ar[dl]|{\Buni{AM}{N}} & \\
\mc{A}(\mc{M}\mc{N})\ar[rr]_{a^2} & & (\mc{A}\mc{M})\mc{N} \ar[rr]_{a^1} & & \mc{A}(\mc{M}\mc{N})
} \]
The triangles in upper left and right are those defining $\alpha^2$, $\alpha^1$ respectively. The central square is the definition of $\Buni{A}{M}\tens{D} id_\mc{N}$, and the left and right squares those defining $a^2$ and $a^1$. Thus the perimeter commutes, giving
\begin{eqnarray*}
& &a^1a^2\Buni{A}{MN}(id_\mc{A}\boxtimes\Buni{M}{N}) = (id_\mc{A}\tens{C}\Buni{M}{N})B_{\mc{A}, \mc{M}\boxtimes\mc{N}}\\
&\Rightarrow& a^1a^2\Buni{A}{MN}(id_\mc{A}\boxtimes\Buni{M}{N}) = \Buni{A}{MN}(id_\mc{A}\boxtimes\Buni{M}{N})\\
&\Rightarrow& a^1a^2\Buni{A}{MN}(\alpha^2)^{-1}B_{\mc{A}\boxtimes\mc{M}, \mc{N}} = \Buni{A}{MN}(\alpha^2)^{-1}B_{\mc{A}\boxtimes\mc{M}, \mc{N}}\\
&\Rightarrow& a^1a^2\Buni{A}{MN} = \Buni{A}{MN}\\
&\Rightarrow& a^1a^2 = id_{\mc{A}(\mc{M}\mc{N})}\\
\end{eqnarray*}
where the first implication follows from the square defining $id_\mc{A}\tens{C}\Buni{M}{N}$, the second by the definition of $\alpha^2$, the third by Lemma \ref{universelemma} (for $B_{\mc{A}\boxtimes\mc{M}, \mc{N}}$, $\Buni{A}{MN}$, resp.). Using a similar diagram one derives $a^2a^1 = id_{(\mc{A}\mc{M})\mc{N}}$ hence the $a^i$ are equivalences and by Remark \ref{lem;unimod} they are bimodule equivalences.
\end{proof}
In what follows denote
\begin{equation}\label{partialAss}
a_{\mc{A}, \mc{M}, \mc{N}} := a^1_{\mc{A}, \mc{M}, \mc{N}} : (\mc{A}\tens{C}\mc{M})\tens{D}\mc{N} \simeq \mc{A}\tens{C}(\mc{M}\tens{D}\mc{N}).
\end{equation}
In order to prove coherence for $a$ (Proposition \ref{asslem:2}) we will need a couple of simple technical lemmas together with results about the naturality of $a$.
In the monoidal category setting associativity of monoidal product is required to be natural in each of its indices, which are taken as objects in the underlying category. In describing monoidal structure in the 2-category setting we also require associativity though stipulate that it be natural in its indices  \textit{up to} 2-isomorphism (see M.10 in Definition \ref{defnMon2}). For us this means, in the first index,
\begin{equation*}
a_{F, \mc{M}, \mc{N}}:a_{\mc{B},\mc{M},\mc{N}}(F\mc{M})\mc{N} \stackrel{\sim}{\Rightarrow}F(\mc{M}\tens{D}\mc{N})a_{\mc{A},\mc{M},\mc{N}}
\end{equation*} 
for bimodule functor $F:\mc{A}\rightarrow\mc{B}$. Similarly we need 2-isomorphisms for $F$ in the remaining positions. The content of Proposition \ref{ass2-natProp} is that all such 2-isomorphims are actually identity. Before stating the proposition we give a definition to introduce a notational convenience.
\begin{defn}\label{FMdef} For right exact right $\mc{C}$-module functor $F:\mc{A}\rightarrow\mc{B}$ define 1-cell $F\mc{M}:=F\tens{C}id_{\mc{M}}:\mc{A}\tens{C}\mc{M}\rightarrow\mc{B}\tens{C}\mc{M}$ and note that $F\mc{M}$ is right exact. Similarly we can act on such functors on the right.
\end{defn}
\begin{prop}[Associativity ``2-naturality"]\label{ass2-natProp} We have
\begin{equation*}
a_{\mc{B},\mc{M},\mc{N}}(F\mc{M})\mc{N} = F(\mc{M}\tens{D}\mc{N})a_{\mc{A},\mc{M},\mc{N}}. 
\end{equation*}
Analogous relations hold for the remaining indexing valencies for $a$. 
\end{prop}
\begin{proof} We will prove the stated naturality of $a$ for 1-cells appearing in the first index. A similar proof with analogous diagrams gives the others. Recall $\alpha^1$ defined in Lemma \ref{prop;asslem}. Consider the diagram:
\[\qquad\quad
\xymatrix{
(\mc{A}\mc{M})\mc{N}\ar[rrrr]^{(F\mc{M})\mc{N}}\ar[dddd]_a&&&&(\mc{B}\mc{M})\mc{N}\ar[dddd]^a\\
&(\mc{A}\mc{M})\boxtimes\mc{N}\ar[ul]|{\Buni{AM}{N}}\ar[rr]^{(F\mc{M})}\ar@/^-2.9pc/[dd]_{\alpha^1_{\mc{A},\mc{M},\mc{N}}}&&(\mc{B}\mc{M})\boxtimes\mc{N}\ar[ur]|{\Buni{BM}{N}}\ar@/^2.9pc/[dd]^{\alpha^1_{\mc{B},\mc{M},\mc{N}}}&\\
&\mc{A}\boxtimes\mc{M}\boxtimes\mc{N}\ar[rr]_{F}\ar[d]^{B_{\mc{A},\mc{M}\boxtimes\mc{N}}}\ar[u]_{\Buni{A}{M}}&&\mc{B}\boxtimes\mc{M}\boxtimes\mc{N}\ar[u]^{\Buni{B}{M}}\ar[d]_{B_{\mc{B}, \mc{M}\boxtimes\mc{N}}}&&\\
&\mc{A}(\mc{M}\boxtimes\mc{N})\ar[rr]_{F(\mc{M}\boxtimes\mc{N})}\ar[dl]|{\mc{A}\Buni{M}{N}}&&\mc{B}(\mc{M}\boxtimes\mc{N})\ar[dr]|{\mc{B}\Buni{M}{N}}&\\
\mc{A}(\mc{M}\mc{N})\ar[rrrr]_{F(\mc{M}\mc{N})}&&&&\mc{B}(\mc{M}\mc{N})
}\]
The top, bottom and center rectangles follow from Definition \ref{FMdef} and definition of tensor product of functors. Commutativity of all other subdiagrams is given in proof of Proposition \ref{prop;ass1}. External contour is the stated relation. 
\end{proof}

\begin{remark}\label{alpha2-natRemark} Observe that the proof of Proposition \ref{ass2-natProp} also gives 2-naturality of $\alpha^1$: the center square with attached arches gives the equation 
\begin{equation}\label{alpha2-natEqn}
\alpha^1_{\mc{B},\mc{M},\mc{N}}((F\mc{M})\boxtimes id_\mc{N})=F(\mc{M}\boxtimes\mc{N})\alpha^1_{\mc{A},\mc{M},\mc{N}}.
\end{equation}
\end{remark}
\begin{lem}\label{asslem:2} The hexagon
\[
\xymatrix{
\mc{A}(\mc{M}\mc{N})\tens{}\mc{P} \ar[d]_{\alpha^1_{\mc{A}, \mc{M}\mc{N}, \mc{P}}} & & (\mc{A}\mc{M})\mc{N}\tens{}\mc{P} \ar[d]^{\Buni{(AM)N}{P}} \ar[ll]_{a_{\mc{A}, \mc{M}, \mc{N}}}\\
\mc{A}(\mc{M}\mc{N}\boxtimes\mc{P}) \ar[d]_{\Buni{MN}{P}}& & ((\mc{A}\mc{M})\mc{N})\mc{P}\ar[d]^{a_{\mc{A}\mc{M}\mc{N}}}\\
\mc{A}((\mc{M}\mc{N})\mc{P})&&(\mc{A}(\mc{M}\mc{N}))\mc{P} \ar[ll]^{a_{\mc{A}, \mc{M}\mc{N},\mc{P}}}\\
} \]
commutes.
\end{lem}

\begin{proof}  
The arrow $\Buni{A(MN)}{P}$ drawn from the upper-left most entry in the hexagon to the lower-right most entry divides the diagram into a pair of rectangles. The upper right rectangle is the definition of $a_{\mc{A}, \mc{M}, \mc{N}}\tens{E}id_{\mc{P}}$ and the lower left rectangle is the definition of $a_{\mc{A}, \mc{M}\mc{N}, \mc{P}}$. 
\end{proof}
In the case of monoidal categories the relevant structure isomorphisms are required to satisfy axioms which take the form of commuting diagrams. In the 2-monoidal case we make similar requirements of the structure morphisms but here, because of the presence of higher dimensional structures, it is necessary to weaken these axioms by requiring only that their diagrams commute up to some 2-morphisms. Above we have defined a 2-associativity isomorphism $a_{\mc{M},\mc{N},\mc{P}}:(\mc{M}\mc{N})\mc{P}\rightarrow\mc{M}(\mc{N}\mc{P})$. In the definition of monoidal 2-category $a$ is required to satisfy the pentagon which appears in the lower dimensional monoidal case, but only up to 2-isomorphism. The content of Proposition \ref{cor;asscor} is that, in the 2-category of bimodule categories, the monoidal structure $\tens{C}$ is \textit{strictly} associative just as it is in the monoidal category setting. For us this means that the 2-isomorphism $a_{\mc{A},\mc{M}, \mc{N},\mc{P}}$ (see M9. Definition \ref{defnMon2}) is identity for any bimodule categories $\mc{A},\mc{M}, \mc{N},\mc{P}$ for which the relevant tensor products make sense.
\begin{prop}[2-associativity hexagon]\label{cor;asscor} The diagram of functors commutes.
\[
\xymatrix{
((\mc{A}\mc{M})\mc{N})\mc{P} \ar[rr]_{a_{\mc{A}\mc{M}, \mc{N}, \mc{P}}}\ar[d]_{a_{\mc{A}, \mc{M}, \mc{N}}\tens{E}\mc{P}} & &(\mc{A}\mc{M})(\mc{N}\mc{P}) \ar[dd]_{a_{\mc{A}, \mc{M}, \mc{N}\mc{P}}} \\
(\mc{A}(\mc{M}\mc{N}))\mc{P} \ar[d]_{a_{\mc{A}, \mc{M}\mc{N}, \mc{P}}} & &\\
 \mc{A}((\mc{M}\mc{N})\mc{P}) \ar[rr]_{\mc{A}\tens{C}a_{\mc{M}, \mc{N}, \mc{P}}} & & \mc{A}(\mc{M}(\mc{N}\mc{P}))\\
} \]
\end{prop}
\begin{proof}
Consider the diagram below. We first show that the faces peripheral to the embedded hexagon commute and then show that the extended perimeter commutes.  
\[
\xymatrix {
(\mc{A}\mc{M})\mc{N}\boxtimes\mc{P}   \ar[dr]|{\Buni{(AM)N}{P}} \ar[rrrr]^{\alpha^1_{\mc{A}\mc{M}, \mc{N}, \mc{P}}}\ar[dd]|{a_{\mc{A}, \mc{M}, \mc{N}}} & & & &(\mc{A}\mc{M})(\mc{N}\boxtimes\mc{P}) \ar[dddd]|{a_{\mc{A}, \mc{M}, \mc{N}\boxtimes\mc{P}}} \ar[dl]|{id_{\mc{A}\mc{M}}\tens{D}\Buni{N}{A}}\\
& ((\mc{A}\mc{M})\mc{N})\mc{P} \ar[rr]_{a_{\mc{A}\mc{M}, \mc{N}, \mc{P}}}\ar[d]|{a_{\mc{A}, \mc{M}, \mc{N}}\tens{E}id_\mc{P}} & &(\mc{A}\mc{M})(\mc{N}\mc{P}) \ar[dd]|{a_{\mc{A}, \mc{M}, \mc{N}\mc{P}}} &\\
\mc{A}(\mc{M}\mc{N})\boxtimes\mc{P} \ar[dd]|{\alpha^1_{\mc{A}, \mc{M}\mc{N}, \mc{P}}}& (\mc{A}(\mc{M}\mc{N}))\mc{P} \ar[d]|{a_{\mc{A}, \mc{M}\mc{N}, \mc{P}}} &  & &\\
& \mc{A}((\mc{M}\mc{N})\mc{P}) \ar[rr]_{id_\mc{A}\tens{C}a_{\mc{M}, \mc{N}, \mc{P}}} & & \mc{A}(\mc{M}(\mc{N}\mc{P})) &\\
\mc{A}(\mc{M}\mc{N}\boxtimes\mc{P}) \ar[ur]|{id_\mc{A}\tens{C}\Buni{MN}{P}} \ar[rrrr]_{id_\mc{A}\tens{C}\alpha^1_{\mc{M}\mc{N}\mc{P}}} & &  & & \mc{A}(\mc{M}(\mc{N}\boxtimes\mc{P})) \ar[ul]|{id_\mc{A}\tens{C}(id_\mc{M}\tens{D}\Buni{N}{P})}\\
} \]
The top rectangle is the definition of $a_{\mc{A}\mc{M}, \mc{N}, \mc{P}}$, the rectangle on the right is naturality of $a$ as in Proposition \ref{ass2-natProp}, the bottom rectangle the definition of $a$ tensored on the left by $\mc{A}$, and the hexagon is Lemma \ref{asslem:2}. To prove commutativity of the extended perimeter subdivide it as indicated below.
\[
\!\!\!\!\!\!\!\!\!
\xymatrix {
(\mc{A}\mc{M})\mc{N}\boxtimes\mc{P}\ar[rrrr]^{\alpha^1_{\mc{A}\mc{M}, \mc{N}, \mc{P}}}\ar[dd]|{a_{\mc{A}, \mc{M}, \mc{N}\tens{}\mc{P}}} & & & &(\mc{A}\mc{M})(\mc{N}\boxtimes\mc{P}) \ar[dddd]|{a_{\mc{A}, \mc{M}, \mc{N}}}\\
& & & (\mc{A}\mc{M})\boxtimes\mc{N}\boxtimes\mc{P} \ar[ulll]|{\Buni{AM}{N}}\ar[ur]|{B_{\mc{A}\mc{M}, \mc{N}\boxtimes\mc{P}}} \ar[dd]^{\alpha^1_{\mc{A}, \mc{M}, \mc{N}\boxtimes\mc{P}}} \ar[dl]|{\alpha^1_{\mc{A}, \mc{M}, \mc{N}}} &\\
\mc{A}(\mc{M}\mc{N})\boxtimes\mc{P} \ar[dd]|{\alpha^1_{\mc{A}, \mc{M}\mc{N}, \mc{P}}}& & \mc{A}(\mc{M}\boxtimes\mc{N})\boxtimes\mc{P} \ar[ll]_{id_\mc{A}\tens{C}\Buni{M}{N}} \ar[dr]|{\alpha^1_{\mc{A}, \mc{M}\boxtimes\mc{N}, \mc{P}}}& &\\
& & & \mc{A}(\mc{M}\boxtimes\mc{N}\boxtimes\mc{P})\ar[dlll]|{id_\mc{A}\tens{C}(\Buni{M}{N}\boxtimes\mc{P})} \ar[dr]|{id_\mc{A}\tens{C}B_{\mc{M}, \mc{N}\boxtimes\mc{P}}} &\\
\mc{A}(\mc{M}\mc{N}\boxtimes\mc{P}) \ar[rrrr]_{id_\mc{A}\tens{C}\alpha^1_{\mc{M}\mc{N}\mc{P}}}& & & &\mc{A}(\mc{M}(\mc{N}\boxtimes\mc{P}))
} \]
The upper and lower triangles are the definitions of $\alpha^1_{\mc{A}\mc{M}, \mc{N}, \mc{P}}$ and $\mc{A}\tens{*}($definition of $\alpha^1_{\mc{M},\mc{N},\mc{P}})$, respectively (Lemma \ref{prop;asslem}). Right rectangle is definition of $a_{\mc{A}, \mc{M}, \mc{N}\boxtimes\mc{P}}$. Upper left rectangle is $($definition of $a_{\mc{A}, \mc{M}, \mc{N}})\boxtimes\mc{P}$, and the lower left rectangle is explained in Remark \ref{alpha2-natRemark}. The central triangle is an easy exercise. An application of Lemma \ref{universelemma} gives the result.
\end{proof}
Let $\mc{M}_i$ be a $(\mc{C}_{i-1}, \mc{C}_i)$-bimodule category tensor categories $\mc{C}_i$ $0\leq i \leq n+1$. Then one extends the arguments above to completely balanced functors (Definition \ref{DEF;multbal}) of larger index to show that any meaningful arrangement of parentheses in the expression $\mc{M}_1\boxtimes_{\mc{C}_1}\mc{M}_2 \cdots \boxtimes_{\mc{C}_{n-1}} \mc{M}_n$ results in an equivalent bimodule category.
\begin{remark}\label{StasheffRemark}
Proposition \ref{cor;asscor} implies that the 2-morphism described in M9 of Definition \ref{M8} is actually identity. The primary polytope associated to associativity in the monoidal 2-category setting is the Stasheff polytope which commutes in this case. It is obvious that the modified tensor product $\hat{\otimes}$ with associativity (\cite{KV} \S 4) is identity and that nearly every face commutes strictly. The two non-trivial remaining faces (one on each hemisphere) agree trivially. We refer the reader to the original paper for details and notation. 
\end{remark}
\subsection{Unit constraints}
Recall from Proposition \ref{lem;unit} the equivalences $l_{\mc{M}}:\bitens{M}{D}{D}\simeq \mc{M}$ and $r_{\mc{M}}:\bitens{C}{C}{M}\simeq \mc{M}$. The first proposition of this section explains how $l, r$ interact with $2$-associativity $a$.
\begin{prop}\label{cor;unitmaps} $(id_{\mc{M}}\tens{D}l_{\mc{N}})a_{\mc{M}, \mc{N}, \mc{E}} = l_{\mc{M}\tens{D}\mc{N}}$, $r_{\mc{M}\tens{D}\mc{N}}(a_{\mc{C}, \mc{M}, \mc{N}}) = r_{\mc{M}}\tens{D}id_{\mc{N}}$. Also the triangle
\[
\xymatrix{
(\bitens{M}{D}{D})\tens{D}\mc{N}\ar[dr]^{\;\;\;l_{\mc{M}}\tens{D}id_\mc{N}}\ar[d]_{a_{\mc{M}, \mc{D},\mc{N}}}&\\
\mc{M}\tens{D}(\bitens{D}{D}{N})\ar[r]_>>>>>{id_\mc{M}\tens{D}r_{\mc{N}}}&\bitens{M}{D}{N}
}\]
commutes up to a natural isomorphism.
\end{prop}
\begin{proof}
The first two statements follow easily from definitions of $\alpha^1$ (Lemma \ref{prop;asslem}), module structure in $\mc{M}\boxtimes_{D}\mc{N}$ and those of $l$ and $r$. This means that the 2-isomorphisms $\rho$ and $\lambda$ in M11 of Defintition \ref{M8} are both trivial. 

The diagram below relating $l$ and $r$ commutes only up to balancing isomorphism $b$ for $B_{\mc{M}, \mc{N}}$ where we write $b:B_{\mc{M}, \mc{N}}(\otimes\boxtimes id_{\mc{N}})\Rightarrow B_{\mc{M}, \mc{N}}(id_{\mc{M}}\boxtimes\otimes)$. 
All juxtaposition takes place over $\mc{D}$.
\[
\xymatrix{
(\bitens{M}{}{D})\mc{N} \ar[rrrrr]^{\alpha^1_{\mc{M},\mc{D},\mc{N}}} \ar[dddd]_{\Buni{M}{D}}\ar[dddrr]_>>>>>>>>{}="1"_{\otimes} &&&&& \mc{M}\tens{}(\mc{D}\mc{N})\ar[dddd]^{\Buni{M}{CN}}\\
&&\mc{M}\tens{}\mc{D}\tens{}\mc{N}\ar[ull]_{B_{\bitens{M}{}{D}, \mc{N}}}\ar[urrr]^{\Buni{D}{N}}\ar[dr]^{B_{\mc{M}, \bitens{D}{}{N}}}&&&\\
&&&\mc{M}(\bitens{D}{}{N})\ar[dl]_>>>>>>>>>{\,}="2"^{\otimes}\ar[ddrr]^{\Buni{D}{N}}&\\
&&\mc{M}\mc{N}&&&\\
(\mc{M}\mc{D})\mc{N}\ar[rrrrr]_{a_{\mc{M}, \mc{D},\mc{N}}}\ar[urr]_{l_{\mc{M}}}&&&&&\mc{M}(\mc{D}\mc{N})\ar[ulll]_{r_{\mc{N}}}\\
\ar@/_-.8pc/@<-.3ex>@{=>}"1";"2"^{b}
}
\]
Top triangle is definition of $\alpha^1$, rectangle is definition of $id_{\mc{M}}\tens{D}\Buni{D}{N}$, lower right triangle is $\mc{M}\tens{D}($definition of $r_{\mc{N}})$, triangle on left is $($definition of $l_{\mc{M}})\tens{D}\mc{N}$, and central weakly commuting rectangle is definition of balancing $b$ for $\Buni{M}{N}$. The perimeter is a diagram occuring in the proof of Proposition \ref{prop;ass1} (we have been sloppy with the labeling of the arrow across the top). Since all other non-labeled faces commute we may write, after chasing paths around the diagram,
\begin{equation*}
l_{\mc{M}}\tens{D}id_\mc{N}(\Buni{M}{D}\tens{D}id_\mc{N})B_{\bitens{M}{}{D}, \mc{N}}\stackrel{b}{\simeq} (id_\mc{M}\tens{D}r_{\mc{N}})a_{\mc{M}, \mc{D},\mc{N}}(\Buni{M}{D}\tens{D}id_\mc{N})B_{\bitens{M}{}{D}, \mc{N}}.
\end{equation*}
Applying Lemma \ref{universelemma} twice we obtain a unique natural isomorphism 
\begin{equation}\label{the mu}
\mu_{\mc{M},\mc{N}}:l_{\mc{M}}\tens{D}id_\mc{N}\stackrel{\sim}{\rightarrow} (id_\mc{M}\tens{D}r_{\mc{N}})a_{\mc{M}, \mc{D},\mc{N}}
\end{equation}
having the property that $\mu_{\mc{M},\mc{N}}*((\Buni{M}{D}\tens{D}id_\mc{N})B_{\bitens{M}{}{D}, \mc{N}})=b$, the balancing in $B_{\mc{M}, \mc{N}}$.
\end{proof}
%
\section{Proof of Theorem \ref{thm:main}}\label{section:2cat}
In this section we finish verifying that the list of requirements given in the definition of monoidal 2-category (\cite{KV}), Definition \ref{defnMon2} of this paper, are substantiated by the scenario where we take as underlying 2-category $\mc{B}(\mc{C})$. Recall that for a fixed monoidal category $\mc{C}$ the 2-category $\mc{B}(\mc{C})$ is defined as having 0-cells $\mc{C}$-bimodule categories, 1-cells $\mc{C}$-bimodule functors and 2-cells monoidal natural transformations. M1-M11 are evident given what we have discussed so far; explicitly, and in order, these are given in Proposition \ref{lem;unit}, Proposition \ref{lem;tensbimd}, Definition \ref{FMdef}, Remark \ref{tensorOf2-cells} (take one of the 2-cells to be identity transformation on identity functor), Equation \ref{partialAss}, Proof of Proposition \ref{lem;unit}, Definition \ref{FMdef} (trivial, composition with id commutes), Polytope \ref{polytope 2}, Proposition \ref{cor;asscor} (trivial), Proposition \ref{ass2-natProp} ($a_{F, \mc{M}, \mc{N}}=id$ for bimodule functor $F$), Proof of Proposition \ref{cor;unitmaps}. 
Commutativity of the Stasheff polytope follows from Proposition \ref{cor;asscor} (see Remark \ref{StasheffRemark}). 

The data introduced throughout are required to satisfy several commuting polytopes describing how they are to interact. Fortunately for us only a few of these require checking since many of the structural morphisms above are identity. Because of this we prove below only those verifications which are not immediately evident. Recall (Definition \ref{FMdef}) that we define action $\mc{M}F$ of bimodule category $\mc{M}$ on module functor $F$.
\begin{polyt}\label{polytope 1} For $F:\mc{M}\rightarrow \mc{M}'$ a morphism in $\mc{B}(\mc{C})$ and any $\mc{C}$-bimodule category $\mc{N}$ the polytopes
\[
\xymatrix{
(\mc{M}\mc{C})\mc{N} \arb{rr}_<<<<<<{}="3"^{a} \ar[dr]^<<<<{}="2"^<<<<<<{}="4"_>>>>{l_{\mc{M}}} \arb{dd}_<<<<{}="1"_{(F\mc{C})\mc{N}} & &\mc{M}(\mc{C}\mc{N}) \ar[dl]^{r_{\mc{N}}} \arb{dd}^{F(\mc{C}\mc{N})}\\
&\mc{M}\mc{N}\ar[dd]^<<<<<<{F\mc{N}} & \\
(\mc{M}'\mc{C})\mc{N} \ar[rr]|{\,\,}_<<<<<<{}="5"^<<<<<<<<a \arb{dr}^<<<<<<{}="6"_{l_{\mc{M}'}} & &\mc{M}'(\mc{C}\mc{N})\arb{dl}^{r_{\mc{N}}}\\
& \mc{M}'\mc{N} &
\ar@/_.6pc/@{=>}"1";"2"_{l_F\mc{N}}
\ar@/_.6pc/@{=>}"4";"3"_{\mu_{\mc{M},\mc{N}}}
\ar@/_.6pc/@{=>}"6";"5"_{\mu_{\mc{M}',\mc{N}}}
}\;
\xymatrix{
(\mc{N}\mc{C})\mc{M} \arb{rr}_<<<<<<{}="3"^{a} \ar[dr]^<<<<<<{}="4"_{l_{\mc{N}}} \arb{dd}_{(\mc{N}\mc{C})F} & &\mc{N}(\mc{C}\mc{M}) \ar[dl]^<<<<{}="2"^>>>>{r_{\mc{M}}} \arb{dd}_<<<<{}="1"^{\mc{N}(\mc{C}F)}\\
&\mc{N}\mc{M}\ar[dd]^<<<<<<{\mc{N}F} & \\
(\mc{N}\mc{C})\mc{M'}\ar[rr]|{\,\,}_<<<<<<{}="5"^<<<<<<<<a \arb{dr}^<<<<<<{}="6"_{l_{\mc{N}}} & &\mc{N}(\mc{C}\mc{M}')\arb{dl}^{r_{\mc{M}'}}\\
& \mc{N}\mc{M}' &
\ar@/^.6pc/@{=>}"1";"2"^{\mc{N}r_F}
\ar@/_.6pc/@{=>}"4";"3"_{\mu_{\mc{N},\mc{M}}}
\ar@/_.6pc/@{=>}"6";"5"_{\mu_{\mc{N},\mc{M}'}}
}\]
commute. Similarly there are commuting prisms for upper left vertex corresponding to the remaining four permutations of $\mc{M}, \mc{C}, \mc{N}$ with upper and lower faces commuting up to either $\lambda$ or $\rho$. 
\end{polyt}
In \cite{KV} these triangular prisms are labeled $(\rightarrow\otimes 1\otimes\bullet), (1\otimes\rightarrow\otimes\bullet)$, etc.
\begin{proof}
We verify commutativity of the second polytope. Commutativity of the other prisms is proved similarly. Denote by $*$ mixed composition of cells. Commutativity of polytope on the right is equivalent to the equation 
\begin{equation}\label{PolytopeEquation}
(id_{\mc{N}}\tens{C}\overline{f})((id_{\mc{N}}\tens{C}F)*\mu_{\mc{N}, \mc{M}}) = \mu_{\mc{N}, \mc{M}'}*(id_{\mc{N}\tens{C}\mc{C}}\tens{C}F)
\end{equation}
where $f$ is module structure of $F$ and $\overline{f}=r_{F}$ (recall Corollary \ref{Functoriality l, r lemma}). Let LHS and RHS denote the left and right sides of (\ref{PolytopeEquation}). Then one easily shows that both LHS$*((B_{\mc{N}, \mc{C}\tens{C}}id_{\mc{M}})B_{\tens{N}{}{C}, \mc{M}})_{M\tens{}X\tens{}N}$ and RHS$*((B_{\mc{N}, \mc{C}\tens{C}}id_{\mc{M}})B_{\tens{N}{}{C}, \mc{M}})_{M\tens{}X\tens{}N}$, for $N\in\mc{N}, X\in\mc{C}, M\in\mc{M}$, are equal to $b'_{N, X, F(M)}$ where $b'$ is the balancing for $B_{\mc{N}, \mc{M}'}$. Two applications of Lemma \ref{universelemma} now imply that LHS=RHS.
\end{proof}
The next polytope concerns functoriality of the 2-cells $l_F, r_F$.
\begin{polyt}\label{polytope 2} Let $\mc{M}\stackrel{F}{\rightarrow}\mc{N}\stackrel{G}{\rightarrow}\mc{P}$ be composible 1-morphisms in $\mc{B}(\mc{C})$. Then the prisms
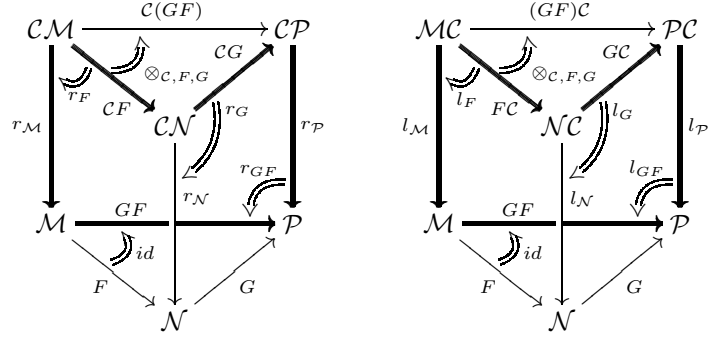
\[
\xymatrix{
\mc{C}\mc{M} \ar[rr]_<<<<<<<<{}="3"^{\mc{C}(GF)} \arb{dr}^<<<<{}="2"^<<<<<<{}="4"_>>>>{\mc{C}F} \arb{dd}_<<<<{}="1"_{r_{\mc{M}}} & &\mc{C}\mc{P}\arb{dd}_>>>>{}="a"^{r_{\mc{P}}}\\
&\mc{C}\mc{N}\ar[dd]^<<<<<<{}="w"^<<<<<<<<{r_{\mc{N}}} \arb{ur}_<<<<{}="z"^>>>>>{\mc{C}G}& \\
\mc{M} \ar@<-.15ex>[rr]|{\,\,}\ar@<.15ex>[rr]|{\,\,}\ar@<-.075ex>[rr]|{\,\,}\ar@<.075ex>[rr]|{\,\,}\ar[rr]|{\,\,}_>>>>{}="b"_<<<<<<{}="y"^<<<<<<<<{GF} \ar[dr]^<<<<<<{}="x"^<<<<<<{}="6"_{F} & &\mc{P}\\
& \mc{N} \ar[ur]_{G}&
\ar@/^.6pc/@{=>}"2";"1"^{r_F}
\ar@/_.6pc/@{=>}"4";"3"_{\otimes_{\mc{C}, F, G}}
\ar@/_.6pc/@{=>}"x";"y"_{id}
\ar@/^.6pc/@{=>}"z";"w"^<<{r_{G}}
\ar@/_.6pc/@{=>}"a";"b"_<<{r_{GF}}
}\;\quad\quad
\xymatrix{
\mc{M}\mc{C} \ar[rr]_<<<<<<<<{}="3"^{(GF)\mc{C}} \arb{dr}^<<<<{}="2"^<<<<<<{}="4"_>>>>{F\mc{C}} \arb{dd}_<<<<{}="1"_{l_{\mc{M}}} & &\mc{P}\mc{C}\arb{dd}_>>>>{}="a"^{l_{\mc{P}}}\\
&\mc{N}\mc{C}\ar[dd]^<<<<<<{}="w"^<<<<<<<<{l_{\mc{N}}} \arb{ur}_<<<<{}="z"^>>>>>{G\mc{C}}& \\
\mc{M} \ar@<-.15ex>[rr]|{\,\,}\ar@<.15ex>[rr]|{\,\,}\ar@<-.075ex>[rr]|{\,\,}\ar@<.075ex>[rr]|{\,\,}\ar[rr]|{\,\,}_>>>>{}="b"_<<<<<<{}="y"^<<<<<<<<{GF} \ar[dr]^<<<<<<{}="x"^<<<<<<{}="6"_{F} & &\mc{P}\\
& \mc{N} \ar[ur]_{G}&
\ar@/^.6pc/@{=>}"2";"1"^{l_F}
\ar@/_.6pc/@{=>}"4";"3"_{\otimes_{\mc{C}, F, G}}
\ar@/_.6pc/@{=>}"x";"y"_{id}
\ar@/^.6pc/@{=>}"z";"w"^<<{l_{G}}
\ar@/_.6pc/@{=>}"a";"b"_<<{l_{GF}}
}\]
commute.
\end{polyt}
\begin{proof}
We prove commutativity of the first prism. Commutativity of the second follows similarly. It is obvious that $\otimes_{\mc{C}, F, G}$ is trivial (it is just composition of functors). First polytope is the condition  $r_{GF} = (G*r_F)(r_G*\mc{C}F)$. Let $f$ be left $\mc{C}$-linearity for $F$, $g$ that for $G$.  Then $(G, g)(F, f) := (GF, g\bullet f)$ where $(g\bullet f)_{X, M} = g_{X, F(M)}G(f_{X, M})$ is left $\mc{C}$-linearity for $GF$. 
One checks directly that 
\begin{equation*}
(G*r_F)(r_G*\mc{C}F)*B_{\mc{C}, \mc{M}}
= (g\bullet f)^{-1}.
\end{equation*}
$r_{GF}$ is defined as the unique 2-isomorphism for which $r_{GF}*\Buni{C}{M} = (g\bullet f)^{-1}$ so Lemma \ref{universelemma} gives the result. 
\end{proof}
\begin{polyt}\label{polytope 3} For any 2-cell $\alpha:F\Rightarrow G$ in $\mc{B}(\mc{C})$  the cylinders
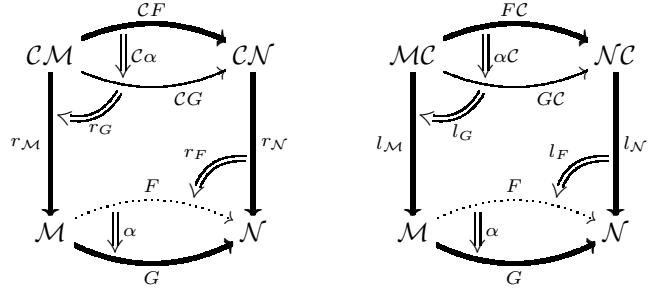
\[
\xymatrix{
\mc{C}\mc{M}\ar@/^1pc/@<-.15ex>[rr]\ar@/^1pc/@<.15ex>[rr]\ar@/^1pc/@<-.075ex>[rr]\ar@/^1pc/@<.075ex>[rr]\ar@/^1pc/[rr]^{\mc{C}F}_<<<<{}="1" \ar@/_1pc/[rr]_<<<<{}="2"_<<<<{}="3"_>>>>{\mc{C}G} \arb{dd}_{r_{\mc{M}}}^<<<<<<{}="4"& &\mc{C}\mc{N}\arb{dd}^{r_{\mc{N}}}_>>>>>>>>{}="7"\\
&&&\\
\mc{M}\ar@/_1pc/@<-.15ex>[rr]\ar@/_1pc/@<.15ex>[rr]\ar@/_1pc/@<-.075ex>[rr]\ar@/_1pc/@<.075ex>[rr]\ar@/_1pc/[rr]_G_<<<<{}="6" \ar@/^1pc/@{.>}[rr]^>>>>{}="8"_<<<<{}="5"^F& &\mc{N}\\
\ar@{=>}"1";"2"^{\mc{C}\alpha}
\ar@/^.6pc/@{=>}"3";"4"^{r_{G}}
\ar@{=>}"5";"6"^{\alpha}
\ar@/_.6pc/@{=>}"7";"8"_{r_{F}}
}
\;\hfill
\xymatrix{
\mc{M}\mc{C}\ar@/^1pc/@<-.15ex>[rr]\ar@/^1pc/@<.15ex>[rr]\ar@/^1pc/@<-.075ex>[rr]\ar@/^1pc/@<.075ex>[rr]\ar@/^1pc/[rr]^{F\mc{C}}_<<<<{}="1" \ar@/_1pc/[rr]_>>>>{G\mc{C}}_<<<<{}="2"_<<<<{}="3" \arb{dd}_{l_{\mc{M}}}^<<<<<<{}="4"& &\mc{N}\mc{C}\arb{dd}^{l_{\mc{N}}}_>>>>>>>>{}="7"\\
&&&\\
\mc{M}\ar@/_1pc/@<-.15ex>[rr]\ar@/_1pc/@<.15ex>[rr]\ar@/_1pc/@<-.075ex>[rr]\ar@/_1pc/@<.075ex>[rr]\ar@/_1pc/[rr]_{G}_<<<<{}="6" \ar@/^1pc/@{.>}[rr]^>>>>{}="8"_<<<<{}="5"^F& &\mc{N}\\
\ar@{=>}"1";"2"^{\alpha\mc{C}}
\ar@/^.6pc/@{=>}"3";"4"^{l_{G}}
\ar@{=>}"5";"6"^{\alpha}
\ar@/_.6pc/@{=>}"7";"8"_{l_{F}}
}\]
commute.
\end{polyt}
\begin{proof} Again we sketch commutativity of the first polytope and leave the second to the reader. The first cylinder is the condition $(\alpha *r_{\mc{M}})r_F = r_G(r_{\mc{N}}*\mc{C}\alpha)$ 
where $\mc{C}\alpha$ is the 2-cell defined by $\textbf{id}_{\mc{C}}\tens{C}\alpha$ and $\textbf{id}_\mc{C}$ means natural isomorphism $id:id_\mc{C}\Rightarrow id_{\mc{C}}$. One verifies this directly using bimodule condition on $\alpha$.
\end{proof}
This completes verification of the polytopes required for monoidal 2-category structure, and therefore completes the proof of Theorem \ref{thm:main}.
%
%
%
%
%
%
%
%
%
%
\section{De-equivariantization and tensor product over braided categories}\label{DeEquiSect}
\subsection{Tensor product over braided categories}%
In this section we discuss the tensor structure of $\mc{C}_1\tens{E}\mc{C}_2$ where $\mc{C}_i$ are tensor categories and $\mc{E}$ is a braided tensor category. Recall that we denote by $Z(\mc{C})$ the center of tensor category $\mc{C}$.
\begin{defn}[\cite{BraidedI}]\label{OverDef}
Let $\mc{C}$ be a tensor category. Then we say $\mc{C}$ is a tensor category \textit{over} braided tensor category $\mc{E}$ whenever there is a braided tensor functor $\mc{E}\rightarrow Z(\mc{C})$. In general we will identify objects in $\mc{E}$ with their images in $\mc{C}$ and talk about $\mc{E}$ as a subcategory of $\mc{C}$. 
\end{defn}
Suppose that $\sigma:\mc{E}\rightarrow Z(\mc{C})$ is a braided tensor functor. $\sigma$ gives $\mc{C}$ the structure of an $\mc{E}$-bimodule category via
\begin{equation}
X\otimes M := \Forg(\sigma(X)\otimes M)
\end{equation}
where the tensor product is that in $Z(\mc{C})$ and $\Forg:Z(\mc{C})\rightarrow\mc{C}$ is forgetful functor. $\mc{E}$ right acts on $\mc{C}$ via $(M, X)\mapsto X\otimes M$ giving $\mc{C}$ the structure of $\mc{E}$-bimodule category (see Proposition \ref{leftbimodulelemma}). Now let $\mc{C}_i$ for $i=1, 2$ be tensor categories over braided tensor $\mc{E}$. Since $\mc{C}_i$ are $\mc{E}$-bimodules we can form their tensor product $\mc{C}_1\tens{E}\mc{C}_2$.

\begin{thm}\label{de-equiGenProp} There is a canonical tensor category structure on $\mc{C}_1\tens{E}\mc{C}_2$ such that the universal balanced functor $B_{\mc{C}_1, \mc{C}_2}:\mc{C}_1\tens{}\mc{C}_2\rightarrow\mc{C}_1\tens{E}\mc{C}_2$ is tensor.
\end{thm}
\begin{proof}
Denote by $\Lambda$ the composition of functors
\[
\xymatrix{
\Lambda:=\mc{C}_1\tens{}\mc{C}_2\tens{}\mc{C}_1\tens{}\mc{C}_2\ar[r]^<<<<<{\tau_{2, 3}}&\mc{C}_1\tens{}\mc{C}_1\tens{}\mc{C}_2\tens{}\mc{C}_2\ar[r]^<<<<{(\otimes_1, \otimes_2)}&\mc{C}_1\tens{}\mc{C}_2\ar[r]^<<<<{B_{\mc{C}_1, \mc{C}_2}}&\mc{C}_1\tens{E}\mc{C}_2.
}\]
$\tau_{2, 3}$ permutes second and third entries. Let $b$ denote the balancing for functor $B_{\mc{C}_1, \mc{C}_2}$. Let $X\in\mc{E}, M_i\in\mc{C}_1, N_i\in\mc{C}_2$. Then it is not difficult to check that $\Lambda$ is $\mc{E}$-balanced in positions $1, 3$ (Definition \ref{DEF;multbal}) with balancing morphisms 
\begin{eqnarray*}
b^1
:\Lambda((M_1\otimes X)\tens{}N_1\tens{}M_2\tens{}N_2)\rightarrow \Lambda(M_1\tens{}(X\otimes N_1)\tens{}M_2\tens{}N_2)\\
b^3
:\Lambda(M_1\tens{}N_1\tens{}(M_2\otimes X)\tens{}N_2)\rightarrow \Lambda(M_1\tens{} N_1\tens{}M_2\tens{}(X\otimes N_2))
\end{eqnarray*}
given by
\begin{eqnarray*}
b^1_{M_1, X, N_1\tens{}M_2\tens{}N_2}:=b_{M_1\otimes M_2, X, N_1\otimes N_2}\circ((id_{M_1}\otimes c_{X, M_2})\tens{E}id_{N_1\otimes N_2})\\
b^3_{M_1\tens{}N_1\tens{}M_2, X, N_2}:=(id_{M_1\otimes M_2}\tens{E}(c_{X, N_1}\otimes id_{N_2}))\circ b_{M_1\otimes M_2, X, N_1\otimes N_2}.
\end{eqnarray*}
where $c$ is braiding in $Z(\mc{C}_i)$. Diagram \ref{DEF;bal} is not difficult to write down for the $b^i$ and this we leave to the reader. It is also evident that $\Lambda$ is right exact. Thus we get unique right exact $\overline{\Lambda}$:
\[
\xymatrix{
\mc{C}_1\tens{}\mc{C}_2\tens{}\mc{C}_1\tens{}\mc{C}_2\ar[d]_{B^{1, 3}_{\mc{C}_1, \mc{C}_2}}\ar[drr]^{\Lambda}&&\\
(\mc{C}_1\tens{E}\mc{C}_2)\tens{}(\mc{C}_1\tens{E}\mc{C}_2)\ar[rr]_{\overline{\Lambda}}&&\mc{C}_1\tens{E}\mc{C}_2
}
\]
Here $B^{1, 3}_{\mc{C}_1, \mc{C}_2}=B_{\mc{C}_1, \mc{C}_2}\tens{}B_{\mc{C}_1, \mc{C}_2}$ is the universal functor for right exact functors balanced in positions 1, 3 from the abelian category at the apex. Associativity of the tensor product $\overline{\Lambda}$ comes from associativity constraints $a^i$ in $\mc{C}_i$. One shows 
\begin{equation*}
a^1\tens{B}a^2:\overline{\Lambda}(\overline{\Lambda}\tens{}id_{\mc{C}_1\tens{B}\mc{C}_2})\stackrel{\sim}{\longrightarrow}\overline{\Lambda}(id_{\mc{C}_1\tens{B}\mc{C}_2}\tens{}\overline{\Lambda})
\end{equation*}
is natural isomorphism using an extended version of Lemma \ref{universelemma}, evincing $\overline{\Lambda}$ a bona fide tensor structure on $\mc{C}_1\tens{E}\mc{C}_2$. Observe that unit object for $\overline{\Lambda}$ comes from identity objects of $\mc{C}_i$ in the obvious way. This completes the first statement in Proposition \ref{de-equiGenProp}. The second statement follows from the definition of $\overline{\Lambda}$: indeed $B_{\mc{C}_1, \mc{C}_2}$ is a strict tensor functor:
\begin{equation*}
B_{\mc{C}_1, \mc{C}_2}((X\tens{}Y)\otimes (X'\tens{}Y'))= \overline{\Lambda}(B_{\mc{C}_1, \mc{C}_2}(X\tens{}Y)\tens{}B_{\mc{C}_1, \mc{C}_2}(X'\tens{}Y'))
\end{equation*}
since both are $B_{\mc{C}_1, \mc{C}_2}((X\otimes X')\tens{}(Y\otimes Y'))$.
\end{proof}
\begin{remark}\label{FunTen} Let $\beta:\mc{C}_1\tens{E}\mc{C}_2\rightarrow\underline{Fun}_{\mc{E}}(\mc{C}_1^{op}, \mc{C}_2)$ be the equivalence whose existence is implied by Theorem \ref{prop;exist} (see Remark \ref{uniRemark}). 
Then tensor structure in $\mc{C}_1\tens{E}\mc{C}_2$ described above induces a tensor structure on $\underline{Fun}_{\mc{E}}(\mc{C}_1^{op}, \mc{C}_2)$ via $\beta$ as follows. Let $F, G:\mc{C}_1^{op}\rightarrow\mc{C}_2$ be right exact functors. Define $X_F$ to be the object in $\mc{C}_1\tens{E}\mc{C}_2$ with $\beta(X_F)=F$, and define $X_G$ similarly. Then tensor product $F\odot G$ is the right exact functor $\mc{C}_1^{op}\rightarrow\mc{C}_2$ defined by 
\begin{equation}\label{FunctTensEq}
F\odot G := \beta(X_F\otimes X_G)
\end{equation}
where $\otimes$ is that in $\mc{C}_1\tens{E}\mc{C}_2$. Thus $X_{F\odot G}=X_F\otimes X_G$. Associativity comes from that in $\mc{C}_1\tens{E}\mc{C}_2$.
\end{remark}
\subsection{On de-equivariantization and relative tensor product} The main result of this section is the proof of Theorem \ref{functAlgebra}. We begin with the following lemma.
\begin{lem}\label{lemgood} Let $\mc{C}$, $\mc{D}$ be fusion categories and let $F:\mc{C}\rightarrow\mc{D}$ be a surjective tensor functor. Let $I$ be its right adjoint. Then
\begin{enumerate}
\item\label{1} $I(\textbf{1})$ is an algebra in $Z(\mc{C})$.
\item\label{2} $\mc{D}$ is tensor equivalent to the category $\Mod_{\mc{C}}(I(\unit))$ of right $I(\unit)$-modules in $\mc{C}$.
\item\label{3} The equivalence in (\ref{2}) identifies $F$ with the free module functor $X\mapsto X\otimes I(\unit)$. 
\end{enumerate}
\end{lem}
\begin{proof} To prove (\ref{1}) observe that $\mc{D}$ is a $Z(\mc{C})$-module category with action $X\otimes Y:=F'(X)\otimes Y$ where $F':Z(\mc{C})\rightarrow\mc{D}$ is $F$ composed with functor forgetting central structure. Under this action $\underline{\Hom}(\unit, \unit)=I(\unit)$ (see Definition \ref{inthomdef}) so by Lemma 5 in \cite{O:MC} $I(\unit)$ is an algebra in $Z(\mc{C})$. Note that since $I(\unit)$ is an algebra in $Z(\mc{C})$ we have tensor structure on $\Mod_{\mc{C}}(I(\unit))$: $X\otimes I(\unit)=I(\unit)\otimes X$ so for $I(\unit)$-modules $X, Y$ $X\otimes_{I(\unit)}Y$ makes sense. 
Theorem 1 in the same paper says that  $\Mod_{\mc{C}}(I(\unit))\simeq \mc{D}$ as module categories over $\mc{C}$ via $F$ in (\ref{3}). Observe that \begin{equation*}
F(X)\otimes_{I(\unit)}F(Y)=(X\otimes I(\unit))\otimes_{I(\unit)}(Y\otimes I(\unit))=(X\otimes Y)\otimes  I(\unit) = F(X\otimes Y).
\end{equation*} 
Hence $F:X\mapsto X\otimes I(\unit)$ respects tensor structure. This completes the proof of the lemma.
\end{proof}
In what follows $G$ is a finite group and we write $\mc{E}:=\Rep(G)$, the symmetric fusion category of finite dimensional representations of $G$ in $Vec$. Let $\mc{C}$ be tensor category over $\mc{E}$ (Definition \ref{OverDef}) which we thereby view as a right $\mc{E}$-module category. Let $A$ be the regular representation of $G$. $A$ has the structure of an algebra in $\mc{E}$ and we therefore have the notion of $A$-module in $\mc{C}$. Denote by $\mc{C}_G$ the category Mod$_{\mc{C}}(A)$ of $A$-modules in $\mc{C}$. There is functor Free $:\mc{C}\rightarrow\mc{C}_G$, $X\mapsto X\otimes A$ left adjoint to Forg $:\mc{C}_G\rightarrow \mc{C}$ which forgets $A$-module structure (\cite[\S 4.1.9]{BraidedI}). We are now ready to prove the theorem.
\begin{proof}[Proof of Theorem \ref{functAlgebra}]
Let $F:=B_{\mc{C}, Vec}:\mc{C}\tens{}Vec\rightarrow\mc{C}\boxtimes_{\mc{E}}Vec$ be the canonical surjective right exact functor described in Definition \ref{DEF;tensor} which is tensor by Theorem \ref{de-equiGenProp}, and let $I$ be its right adjoint. Lemma \ref{lemgood} gives us tensor equivalence $\Mod_{\mc{C}}(I(\unit))\simeq\mc{C}\boxtimes_{\mc{E}}Vec$. Denote by $A'$ the image of the regular algebra $A$ in $\mc{E}$ under the composition 
\begin{equation}\label{compoz}
\mc{E}\rightarrow Z(\mc{C})\rightarrow\mc{C}. 
\end{equation}
We claim that $I(\unit)$ is $A'$

Let $X, Y\in\mc{C}$ be in distinct indecomposible $\mc{E}$-module subcategories of $\mc{C}$. Since the indecomposible $\mc{E}$-module subcategories of $\mc{C}$ are respected by $F$ the images of $X, Y$ under $F$ are in distinct $\mc{E}$-module components of $\mc{C}\boxtimes_{\mc{E}}Vec$. Not only does this imply that $F(X)$ and $F(Y)$ are not isomorphic but in fact $\Hom(F(X), F(Y))=0$. Thus if $F(X)$ contains a copy of the unit object $\unit\in\mc{C}\boxtimes_{\mc{E}}Vec$ then $X$ and $\unit\in\mc{C}$ must belong to the same indecomposible $\mc{E}$-module subcategory of $\mc{C}$. Thus any object whose $F$-image contains the unit object must be contained in the image of $\mc{E}$ in $\mc{C}$ under the composition (\ref{compoz}).

Note that the restriction of $F$ to the image of $\mc{E}$ in $\mc{C}$ gives a fiber functor $\mc{E}\rightarrow\mc{E}\boxtimes_{\mc{E}}Vec=Vec$. By \cite[\S 2.13]{BraidedI} the choice of a fiber functor from $\mc{E}$ determines a group $G_F\simeq G$ having the property that $Fun(G_F)$ is regular algebra $A$ in $\Rep(G)$ and as such is canonically isomorpic to $I(\unit)$. Thus we have tensor equivalence $\Mod_{\mc{C}}(A)=\mc{C}_G\simeq\mc{C}\boxtimes_{\mc{E}}Vec$ and the proof is complete.
\end{proof}
\section{Module categories over braided monoidal categories}\label{braidedmodulesection}
%
\subsection{The center of a bimodule category}\label{braidedmodulesection1} In this section we describe a construction which associates to a strict $\mc{C}$-bimodule category $\mc{M}$ a new category having the structure of a $Z(\mc{C})$-bimodule category. Assume $\mc{C}$ to be braided with braiding $c_{X, Y}:X\otimes Y\rightarrow Y\otimes X$ (Definition \ref{braidingdef}). Our first proposition is well known and we provide a proof only for completeness. 
\begin{prop}\label{leftbimodulelemma}  Let $\mc{M}$ be a left $\mc{C}$-module category. Then $\mc{M}$ has canonical structure of $\mc{C}$-bimodule category.
\end{prop} 
\begin{proof} 
We begin with the following lemma. 
\begin{lem}\label{lemmleftright} $\mc{M}$ is right $\mc{C}$-module category via  $(M, X)\mapsto X\otimes M$ where $\otimes$ is left $\mc{C}$-module structure.
\end{lem} 
\begin{proof}
For left module associativity $a$ define natural isomorphism
\begin{equation*}
a'_{M, X, Y} = a_{Y, X, M}(id_{M}\otimes c_{X, Y}): M\otimes (X\otimes Y)\rightarrow (M\otimes X)\otimes Y
\end{equation*}
for $X, Y\in\mc{C}$ and $M\in\mc{M}$. In terms of the left module structure by which $M\otimes X$ is defined $a'_{M, X, Y}=a_{Y, X, M}(c_{X, Y}\otimes id_{M})=(X\otimes Y)\otimes M\rightarrow Y\otimes (X\otimes M)$. We show that $a'$ is module associativity for right module structure. Consider  diagram
\[
\!\!\!\!
\xymatrix{
(XYZ)M\ar[rr]^{c_{X, YZ}}\ar[dd]_{c_{XY, Z}}\ar[dr]^{c_{Y, Z}}&&(YZX)M\ar[dr]^{c_{Y,Z}}\ar[r]^{a_{YZ, X, M}}&(YZ)(XM)\ar[r]^{c_{Y, Z}}& (ZY)(XM)\ar[dd]^{a_{Z, Y, XM}}\\
&(XZY)M\ar[rr]^{c_{X, ZY}}\ar[dl]_{c_{X, Z}}&&(ZYX)M\ar[ur]_{a_{ZY, X, M}}\ar[d]^{a_{Z, YX, M}}&\\
(ZXY)M\ar[urrr]_{c_{X, Y}}\ar[rr]_{a_{Z, XY, M}}&&Z((XY)M)\ar[r]_{c_{X, Y}}& Z((YX)M)\ar[r]_{a_{Y, X, M}}&Z(Y(XM))
}\]
The upper left rectangle is naturality of $c$, upper right triangle naturality of $a$, leftmost triangle is equation (\ref{eqn:brd2}), triangle in lower half of diagram is equation (\ref{eqn:brd1}), central bottom rectangle is naturality of $a$ and rightmost rectangle is $a$-pentagon in $\mc{C}$. The two directed components of the external contour are precisely $a'_{MX, Y, Z}a'_{M, X, YZ}$ and $(a'_{M, X, Y}\otimes Z)a'_{M, XY, Z}$. The diagrams for action of unit in $\mc{C}$ are even easier.
\end{proof}
Define action of $X\boxtimes Y\in \mc{C}\boxtimes \mc{C}^{rev}$ using left and right actions, i.e. $(X\boxtimes Y)\otimes M = Y\otimes (X\otimes M)$. Define 
\begin{equation*}
\gamma_{X, M, Y} = a_{X, Y, M}(c_{Y, X}\otimes id_M)a^{-1}_{Y, X, M}: Y\otimes (X\otimes M)\rightarrow X\otimes (Y\otimes M).
\end{equation*}
In order to verify that the candidate action is indeed bimodule we must show that $\gamma$ satisfies the necessary pentagons (Remark \ref{rmk;bimd}). Commutativity of the first pentagon follows from an examination of the diagram below. 
\[
\!\!\!\!\!\!\!
\xymatrix{
(ZXY)M \ar[dr]^{a_{Z, XY, M}}\ar[rrrr]^{c_{Z, XY}}\ar[ddd]|{a_{ZX, Y, M}}&&&&(XYZ)M\ar[dl]_{a_{XY, Z, M}}\ar[ddd]|{a_{X, YZ, M}} \\
&Z((XY)M)\ar[d]_{a_{X, Y, M}}\ar[rr]^{\gamma_{XY, M, Z}} &&(XY)(ZM) \ar[d]^{a_{X, Y, ZM}}&\\
&Z(X(YM)) \ar[r]^{\gamma_{X, YM, Z}} &X(Z(YM)) \ar[r]^{\gamma_{Y, M, Z}} & X(Y(ZM)) &\\
(ZX)(YM) \ar[r]_{c_{Z, X}} \ar[ur]^{a_{Z, X, YM}} &(XZ)(YM) \ar[ur]|{a_{X, Z, YM}} &(XZY)M \ar[l]^{a_{XZ, Y, M}} \ar[r]_{a_{X, ZY, M}} & X((ZY)M) \ar[ul]|{a_{Z, Y, M}} \ar[r]_{c_{Z, Y}} &X((YZ)M)\ar[ul]_{a_{Y, Z, M}}
}
\]
Every peripheral rectangle is either the definition of $\gamma$ or the module associativity satisfied by $a$. Note that top left vertex can be connected to the lower center vertex by the map $c_{Z, X}\otimes id_{Y\otimes M}$ making commutative rectangle expressing naturality of $a$ in first index. Lower center vertex can be connected to uppermost right vertex by the map $id_X\otimes c_{Z, Y}\otimes id_M$ making commutative rectangle expressing naturality of $a$ in the second index. Commutativity of this new external triangle is (equation (\ref{eqn:brd1}))$\otimes M$. Thus the internal pentagon commutes, and this is precisely the first diagram in Remark  \ref{rmk;bimd}. Commutativity of second pentagon is similar. 
\end{proof}
Next we generalize of the notion of center to module categories. 
\begin{defn}\label{centralmoduledef} Let $\mc{M}$ be a $\mc{C}$-bimodule category. A \textit{central structure on} $\mc{M}$ is a family of isomorphisms $\varphi_{X, M}:X\otimes M\simeq M\otimes X$,  $X\in\mc{C}$, one for each object $M\in\mc{M}$, satisfying the condition 
\[
\xymatrix{
(XY)M\ar[d]_{a^{\ell}_{X, Y, M}}\ar[d]\ar[rrr]^{\varphi_{XY, M}}&&&M(XY)\ar[d]^{a^r_{M, XY}}\\
X(YM)\ar[d]_{X\otimes\varphi_{Y, M}}&&&(MX)Y\\
X(MY)\ar[rrr]_{\gamma_{X, M, Y}}&&&(XM)Y\ar[u]_{\varphi_{X, M}\otimes Y}
}\]
whenever $Y\in\mc{C}$ where $a^{\ell}, a^r$ are left and right module associativity in $\mc{M}$ and $\gamma$ bimodule consistency (Proposition \ref{rmk;bimd}). $\varphi_M$ is called the \textit{centralizing isomorphism} associated to $M$.
\end{defn}
Note that when $\mc{M}$ is strict as a bimodule category the hexagon reduces to
\[
\xymatrix{
XMY\ar[rr]^{\varphi_{X, M}\otimes id_Y}&&MXY\\
&XYM\ar[ul]^{id_X\otimes \varphi_{Y, M}}\ar[ur]_{\varphi_{XY, M}}
}
\]
In what follows assume $\mc{C}$ is strict. 
\begin{defn}\label{defCent} The \textit{center} $\cent{M}{C}$ of $\mc{M}$ over $\mc{C}$ consists of objects given by pairs $(M, \varphi_M)$ where $M\in\mc{M}$ and where $\varphi_M$ is a family of natural isomorphisms such that for $X\in\mc{C}$ $\varphi_{X, M}:X\otimes M\simeq M\otimes X$ satisfying Definition \ref{centralmoduledef}. A morphism from $(M, \varphi_M)$ to $(N, \varphi_N)$ in $\cent{M}{C}$ is a morphism $t:M\rightarrow N$ in $\mc{M}$ satisfying $\varphi_{X, N}(id_X\otimes t) = (t\otimes id_X)\varphi_{X, M}$.
\end{defn}
\begin{note} Definition \ref{defCent} appeared in \cite{GNN} in connection with centers of braided fusion categories. 
\end{note}
\begin{ex}\label{central restriction}
For $\mc{C}$ viewed as a having regular bimodule category structure $\cent{C}{C}=Z(\mc{C})$, the center of $\mc{C}$.
\end{ex}
\begin{defn}\label{centfunctdef} Let $\mc{M}, \mc{N}$ be bimodule categories central over $\mc{C}$. Then $\mc{C}$-bimodule functor $T:\mc{M}\rightarrow\mc{N}$ is called \textit{central} if the diagram
\[
\xymatrix{
T(X\otimes M)\ar[rr]^{f_{X, M}}\ar[d]_{T(\varphi_{X, M})}&&X\otimes T(M)\ar[d]^{\varphi_{X, T(M)}}\\
T(M\otimes X)\ar[rr]_{f_{M, X}}&&T(M)\otimes X
}\]
commutes for all $X\in\mc{C}$, $M\in\mc{M}$, where $\varphi$ denotes centralizing natural isomorphisms in $\mc{M}$ and $\mc{N}$. $f$ is linearity isomorphism for $T$. A \textit{central natural transformation} $\tau:F\Rightarrow G$ for central functors $F, G:\mc{M}\rightarrow \mc{N}$ is a bimodule natural transformation $F\Rightarrow G$ with the additional requirement that, for $X\in\mc{C}, M\in\mc{M}$ the diagram
\[
\xymatrix{
X\otimes F(M)\ar[rr]^{\varphi_{X, F(M)}}\ar[d]_{X\otimes\tau_M}&&F(M)\otimes X\ar[d]^{\tau_M\otimes X}\\
X\otimes G(M)\ar[rr]_{\varphi_{X, G(M)}}&&G(M)\otimes X
}
\]
commutes.
\end{defn}
It is evident that centrality of natural transformations is preserved by vertical (and horizontal) composition, and we thus have a category (indeed a bicategory) $Z(\mc{M}, \mc{N})$ for central bimodule categories $\mc{M}, \mc{N}$ consisting of central functors $\mc{M}\rightarrow \mc{N}$ having morphisms central natural transformations.
\begin{lem}\label{centralmodulelemma} $\Cent{\mc{M}}{\mc{C}}$ is a $\Cent{\mc{C}}{}$-bimodule category.
\end{lem}
\begin{proof} Assume $\mc{M}$ is strict bimodule category. We have left action of $Z(\mc{C})$ on $\cent{M}{C}$ given as follows: for $(X, c_X)\in Z(\mc{C})$ and $(M, \varphi_M)\in\cent{M}{C}$ define $(X, c_X)\otimes (M, \varphi_M)=(X\otimes M, \varphi_{X\otimes M})$ where for $Y\in\mc{C}$
\[
\xymatrix{
\varphi_{Y, X\otimes M}:=Y\otimes X\otimes M\ar[r]^>>>>{c^{-1}_{X, Y}\otimes M}&X\otimes Y\otimes M\ar[r]^{X\otimes\varphi_{Y, M}}&X\otimes M\otimes Y
}\]
so that $X\otimes M\in\Cent{\mc{M}}{\mc{C}}$. Define right action of $Z(\mc{C})$ by $(M, \varphi_M)\otimes (X, c_X)= (M\otimes X, \varphi_{M\otimes X})$ where
\[
\xymatrix{
\varphi_{Y, M\otimes X}:=Y\otimes M\otimes X\ar[r]^>>>>{\varphi_{Y, M}\otimes X}&M\otimes Y\otimes X\ar[r]^{M\otimes c_{Y, X}}&M\otimes X\otimes Y
}\]
putting $M\otimes X\in\cent{M}{C}$. It is easy to check that these actions are consistent in the way required of bimodule action.  
\end{proof}
\begin{prop}\label{centralcenter} $\cent{M}{C}$ has a canonical central structure over $Z(\mc{C})$. 
\end{prop}
\begin{proof} $\varphi_{X, M}:(X\otimes M, \varphi_{X\otimes M})\rightarrow (M\otimes X, \varphi_{M\otimes X})$ is a morphism in $\cent{M}{C}$ as can be seen by the diagram
\[
\xymatrix{
YXM\ar[rrr]^{Y\otimes\varphi_{X, M}}\ar[d]_{c_{Y, X}\otimes M}\ar[drrr]_{\varphi_{YX, M}}&&&YMX\ar[d]^{\varphi_{Y, M}\otimes X}\\
XYM\ar[d]_{X\otimes\varphi_{Y, M}}\ar[drrr]^{\varphi_{XY, M}}&&&MYX\ar[d]^{M\otimes c_{Y, X}}\\
XMY\ar[rrr]_{\varphi_{X, M}\otimes Y}&&&MXY
}\]
Triangles are Definition \ref{centralmoduledef} for $\varphi$ and the square is $\mc{C}$-naturality of $\varphi$.
\end{proof}
\begin{prop}\label{centralproplem} For $\mc{C}$-bimodule category $\mc{M}$ we have canonical $Z(\mc{C})$-bimodule equivalence $\underline{Fun}_{\mc{C}\boxtimes\mc{C}^{op}}(\mc{C}, \mc{M})\simeq\cent{M}{C}$.
\end{prop}
\begin{proof} For simplicity assume $\mc{M}$ is strict as a $\mc{C}$-bimodule category. Define functor $\Delta:\underline{Fun}_{\mc{C}\boxtimes\mc{C}^{op}}(\mc{C}, \mc{M})\simeq\cent{M}{C}$ by sending $F\mapsto (F(1), f^r\circ{f^{\ell}}^{-1})$ where $f^{\ell}_{X}:F(X)\simeq X\otimes F(1)$ and $f^r_{X}:F(X)\simeq F(1)\otimes X$ are left/right module linearity isomorphisms for $F$. The diagram below implies $(F(1), f^r\circ{f^{\ell}}^{-1})\in\cent{M}{C}$:
\[
\xymatrix{
F(1)XY&&F(XY)\ar[ll]_{f^r_{XY}}\ar[dll]^{f^r_{XY}}\ar[drr]_{f^{\ell}_{XY}}\ar[rr]^{f^{\ell}_{XY}}&&XYF(1)\\
F(X)Y\ar[u]^{f^r_X\otimes Y}\ar[rr]_{f^{\ell}_X\otimes Y}&&XF(1)Y&&XF(Y)\ar[ll]^{X\otimes f^r_Y}\ar[u]_{X\otimes f^\ell_Y}
}\]
Left and right triangles are diagrams expressing module linearlity of $F$ and square is bimodularity of $F$ (Remark \ref{bimodfunct}). Inverting all $\ell$ superscripted isomorphisms  gives the diagram required for centrality of $f^r\circ{f^{\ell}}^{-1}$.

To complete definition of functor $\underline{Fun}_{\mc{C}\boxtimes\mc{C}^{op}}(\mc{C}, \mc{M})\rightarrow\cent{M}{C}$ we must define action on natural bimodule transformations. For $\tau:F\Rightarrow G$ a morphism in $\underline{Fun}_{\mc{C}\boxtimes\mc{C}^{op}}(\mc{C}, \mc{M})$ note that $\tau_1: (F(\unit), f^r\circ{f^{\ell}}^{-1})\rightarrow (G(\unit), g^r\circ{g^{\ell}}^{-1})$ is a morphism in $\cent{M}{C}$: indeed, diagram required of $\tau_1$ as central morphism is given by pasting together left/right module diagrams for $\tau$ along the edge $\tau_X:F(X)\rightarrow G(X)$.

We now define quasi-inverse $\Gamma$ for functor $\Delta$. For $M\in\mc{M}$ denote by $F_M$ the functor $\mc{C}\rightarrow \mc{M}$ defined by $F_M(X):=X\otimes M$. Right exactness of $F_M$ follows from (contravariant) left exactness of $\Hom(\score, \underline{\Hom}(M, M))$. Since $\mc{M}$ is a strict $\mc{C}$-bimodule category $F_M$ is strict as a left $\mc{C}$-module functor. For $(M, \varphi_M)\in\cent{M}{C}$ we give $F_M$ the structure of a right $\mc{C}$-module functor via
\begin{equation}\label{functorcenterbimod}
F_M(X)=X\otimes M\stackrel{\varphi_{X, M}}{\longrightarrow}M\otimes X = F_M(\unit)\otimes X
\end{equation}
and with this $F_M$ is $\mc{C}$-bimodule. Define $\Gamma(M, \varphi_M):=F_M$ with the bimodule structure given in (\ref{functorcenterbimod}). It is now trivial to verify that $\Delta\Gamma = id$ and that $\Gamma\Delta$ is naturally equivalent to $id$ via $f^{\ell}$. Finally, it is easy to see that $\Gamma$ is a strict $Z(\mc{C})$-bimodule functor.
\end{proof}
As a corollary we get a well known result which appears for example in \cite{EO:FTC}.
\begin{cor}\label{centerdualitycor} $(\mc{C}\boxtimes\mc{C}^{op})^*_{\mc{C}}\simeq Z(\mc{C})$ canonically as monoidal categories.
\end{cor}
\subsection{The 2-categories $\mc{B}(\mc{C})$ and $Z(\mc{C})$-Mod} Recall that $\mc{B}(\mc{C})$ denotes the category of exact $\mc{C}$-bimodule categories. The main result of this section is Theorem \ref{braidingCenter} giving an equivalence $\mc{B}(\mc{C})\simeq Z(\mc{C})$-Mod. Before we give the first proposition of this subsection recall that $\mc{C}$ has a trivial $Z(\mc{C})$-module category structure given by the forgetful functor. 
\begin{prop}\label{CentreEquivFunct} The 2-functor $\mc{B}(\mc{C})\rightarrow Z(\mc{C})$-Mod given by $\mc{M}\mapsto Z_{\mc{C}}(\mc{M})=\underline{Fun}_{\mc{C}\tens{}\mc{C}^{op}}(\mc{C}, \mc{M})$ is an equivalence with inverse given by $\mc{N}\mapsto\underline{Fun}_{Z(\mc{C})}(\mc{C}^{op}, \mc{N})$.
\end{prop}
\begin{proof} In Proposition \ref{centralproplem} we saw that $Z_{\mc{C}}(\mc{M})$ is a $Z(\mc{C})$-module category whenever $\mc{M}$ is a $\mc{C}$-bimodule category (here module structure is just composition of functors). The category of $Z(\mc{C})$-module functors $\underline{Fun}_{Z(\mc{C})}(\mc{C}^{op}, \mc{N})$ for $Z(\mc{C})$-module category $\mc{N}$ has the structure of a $\mc{C}$-bimodule category with actions 
\begin{equation*}
(F\otimes X)(Z):=F(X\otimes Z),\qquad (Y\otimes F)(Z):=F(Z\otimes Y).
\end{equation*}
To see that $\underline{Fun}_{\mc{C}\tens{}\mc{C}^{op}}(\mc{C}, -)$ and $\underline{Fun}_{Z(\mc{C})}(\mc{C}^{op}, -)$ are quasi-inverses first note that 
\begin{equation}\label{FirstDirectionFunctorsZ}
\underline{Fun}_{Z(\mc{C})}(\mc{C}^{op}, \underline{Fun}_{\mc{C}\tens{}\mc{C}^{op}}(\mc{C}, \mc{N}))\simeq\underline{Fun}_{\mc{C}\tens{}\mc{C}^{op}}(\mc{C}\boxtimes_{Z(\mc{C})}\mc{C}^{op}, \mc{N})\simeq\underline{Fun}_{\mc{C}\tens{}\mc{C}^{op}}(Z(\mc{C})^*_{\mc{C}}, \mc{N})
\end{equation}
as $\mc{C}$-bimodule categories for any bimodule category $\mc{N}$ where we have used equation \ref{eqn:frbns} freely. Theorem 3.27 in \textit{loc. cit.} gives a canonical equivalence $(\mc{C}^*_\mc{M})^*_{\mc{M}}\simeq\mc{C}$ for any (exact) $\mc{C}$-module category $\mc{M}$. In the case that $\mc{M}=\mc{C}$ this and Corollary \ref{centerdualitycor} imply $Z(\mc{C})^*_{\mc{C}}\simeq((\mc{C}\tens{}\mc{C}^{op})^*_{\mc{C}})^*_{\mc{C}}\simeq\mc{C}\tens{}\mc{C}^{op}$. Thus the last category of functors in (\ref{FirstDirectionFunctorsZ}) is canonically equivalent to $\underline{Fun}_{\mc{C}\tens{}\mc{C}^{op}}(\mc{C}\tens{}\mc{C}^{op}, \mc{N})\simeq\mc{N}$.

In the opposite direction we have, for $Z(\mc{C})$-module category $\mc{M}$, 
\begin{equation}\label{SecondDirectionFunctorsZ}
\underline{Fun}_{\mc{C}\tens{}\mc{C}^{op}}(\mc{C}, \underline{Fun}_{Z(\mc{C})}(\mc{C}^{op}, \mc{M}))\simeq\underline{Fun}_{Z(\mc{C})}(\mc{C}^{op}\boxtimes_{\mc{C}\tens{}\mc{C}^{op}}\mc{C}, \mc{M}).
\end{equation}
Note that $\mc{C}^{op}\boxtimes_{\mc{C}\tens{}\mc{C}^{op}}\mc{C}\simeq(\mc{C}\tens{}\mc{C}^{op})^*_{\mc{C}}\simeq Z(\mc{C})$ (Corollary \ref{centerdualitycor}) and thus the last category of functors in (\ref{SecondDirectionFunctorsZ}) is canonically equivalent to $\underline{Fun}_{Z(\mc{C})}(Z(\mc{C}), \mc{M})\simeq\mc{M}$.
\end{proof}
\begin{lem}\label{centoplem} As $Z(\mc{C})$-bimodule categories, $\Cent{\mc{M}^{op}}{\mc{C}}\simeq\Cent{\mc{M}}{\mc{C}}^{op}$.
\end{lem}
\begin{proof}
For $\mc{M}$, $\mc{C}$ as above we have the bimodule equivalences
\begin{eqnarray*} 
\underline{Fun}_{\mc{C}\boxtimes\mc{C}^{op}}(\mc{C}, \mc{M}^{op})\simeq\underline{Fun}_{\mc{C}\boxtimes\mc{C}^{op}}(\mc{M}^{op}, \mc{C})^{op} \simeq\underline{Fun}_{\mc{C}\boxtimes\mc{C}^{op}}(\mc{C}, \mc{M})^{op}.
\end{eqnarray*}
The first equivalence is Lemma \ref{Natadjoint} and the second uses Corollary \ref{cor;frbns}. By Proposition \ref{centralproplem} the first term is equivalent to $\Cent{\mc{M}^{op}}{\mc{C}}$ and the last to $\cent{M}{C}^{op}$.
\end{proof}
\begin{thm}\label{braidingCenter}The 2-equivalence $Z_{\mc{C}}:\mc{B}(\mc{C})\simeq Z(\mc{C})$-Mod is monoidal in that $\Cent{\bitens{M}{C}{N}}{\mc{C}}\simeq\cent{M}{C}\boxtimes_{Z(\mc{C})}\cent{N}{C}$ whenever $\mc{M}, \mc{N}$ are $\mc{C}$-bimodule categories.
\end{thm}
\begin{proof} 
We have seen that $Z_{\mc{C}}(\mc{M})$ is a $Z(\mc{C})$-module category whenever $\mc{M}$ is a $\mc{C}$-bimodule category. We have canonical $Z(\mc{C})$-bimodule equivalences
\begin{eqnarray*}
\Cent{\bitens{M}{C}{N}}{\mc{C}}&\simeq&\underline{Fun}_{\mc{C}\boxtimes\mc{C}^{op}}(\mc{C}, \bitens{M}{C}{N})\simeq\underline{Fun}_{\mc{C}\boxtimes\mc{C}^{op}}(\mc{M}^{op}, \mc{N})\\
&\simeq&\underline{Fun}_{Z(\mc{C})}(\Cent{\mc{M}^{op}}{\mc{C}}, \cent{N}{C})\simeq\underline{Fun}_{Z(\mc{C})}(\Cent{\mc{M}}{\mc{C}}^{op}, \cent{N}{C})\\
&\simeq&\underline{Fun}_{Z(\mc{C})}(Z(\mc{C}), \cent{M}{C}\boxtimes_{Z(\mc{C})}\cent{N}{C})\simeq\cent{M}{C}\boxtimes_{Z(\mc{C})}\cent{N}{C}.
\end{eqnarray*}
The first equivalence is Proposition \ref{centralproplem}, the second and fifth are Corollary  \ref{cor;frbns}, the third follows from the fact that the equivalence of 2-categories $Z(\mc{C})$-Mod $\simeq(\mc{C}\boxtimes\mc{C}^{op})^*_{\mc{C}}$-Mod (Corollary \ref{centerdualitycor}) preserves categories of 1-cells, and the fourth follows from Lemma \ref{centoplem}. Example \ref{central restriction} shows that $Z_{\mc{C}}$ preserves units.
\end{proof}
\begin{cor}
Let $\mc{M}$ be a $\mc{C}$-module category for finite tensor $\mc{C}$. There is a canonical $2$-equivalence $\mc{B}(\mc{C})\simeq\mc{B}(\mc{C}^*_\mc{M})$ respecting monoidal structure.
\end{cor}
\begin{proof} Corollary 3.35 in \cite{EO:FTC} says that $Z(\mc{C})\simeq Z(\mc{C}^*_\mc{M})$. The result follows from Theorem \ref{braidingCenter}.
\end{proof}
\section{Fusion rules for $\Rep(G)$-module categories}\label{sec;ex}
\subsection{The Burnside Ring and Monoidal Structure in $Vec_G$-Mod}\label{first ex sec} Much in the beginning of this section is basic and can be found for example in \cite{CRII}. Let $G$ be a finite group. Recall that the Burnside Ring $\Omega(G)$ is defined to be the commutative ring generated by isomorphism classes of $G$-sets with addition and multiplication given by disjoint union and cartesian product:
\begin{eqnarray*}
\langle H\rangle + \langle K\rangle = G/H\cup G/K\\
\langle H\rangle\langle K\rangle = G/H\times G/K
\end{eqnarray*}
Here $\langle H\rangle$ denotes the isomorphism class of the $G$-set $G/H$ for $H<G$ and $G$ acts diagonally over $\times$. Evidently we have
\begin{equation*}
\langle H\rangle\langle G\rb = \lb H\rb,\quad \lb H\rb\lb 1\rb = [G:H]\lb 1\rb
\end{equation*}
so $\Omega(G)$ is unital with $1 = \lb G\rb$. It is a basic exercise to check that multiplication in $\Omega(G)$ satisfies the equation\footnotemark
\begin{equation*}\label{orbits lemma}
\langle H\rangle\langle K\rangle = \sum_{HaK\in H\setminus G/K}\langle H\cap {}^aK\rangle.
\end{equation*}
\footnotetext{One uses the fact that there is a bijection between the $G$-orbits of $(xH, yK)\in G/H\times G/K$ and double cosets $H\setminus G/K$ given by $(sH, tK)\mapsto Hs^{-1}tK$. The orbit corresponding to the coset $HaK$ contains $(H, aK)$ with stabilizer $H\cap {}^aK$, thus orbit $\mc{O}_G(H, aK)$ of $(H, aK)$ is $G/(H\cap {}^aK)$ as $G$-sets giving the formula.}
We are interested in a twisted variant of the Burnside ring. Here we take as basis elements $\langle H, \sigma\rangle$ where $G/H$ is a $G$-set and $\sigma$ is a $k^{\times}$-valued 2-cocycle on $H$. Multiplication of basic elements takes the form
\begin{equation*}
\langle H, \mu\rangle\langle K, \sigma\rangle = \sum_{HaK\in H\setminus G/K}\langle H\cap{}^aK, \mu\sigma^a\rangle
\end{equation*}
where on the right $\mu, \sigma^a$ refer to restriction to the subgroup $H\cap {}^aK$ from $H, {}^aK$, respectively. The cocycle $\sigma^a:{}^aK\times {}^aK\rightarrow k^{\times}$ is defined by $\sigma^a(x, y) = \sigma(x^a, y^a)$. 

\begin{note} The decomposition for twisted Burnside products described above occurred in \cite{OdaYoshida} in order to study crossed Burnside rings, and in \cite{Rose} in connexion with the extended Burnside ring of semisimple $\Rep(G)$-module categories $\mc{M}$ having exact faithful module functor $\mc{M}\rightarrow\Rep(G)$. 
\end{note}
Recall that indecomposable $Vec_G$-module categories are parametrized by pairs $(H, \mu)$ where $H<G$ and $\mu\in H^2(H, k^{\times})$. Denote module category associated to such a pair by $\mc{M}(H, \mu)$. Explicitly simple objects of $\mc{M}(H, \mu)$ form a $G$-set with stabilizer $H$ and are thus in bijection with cosets in $G/H$. Module associativity is given by scalars $\mu(g_1, g_2)(X)$, for $\mu\in Z^2(G, Fun(G/H, k^{\times}))$, associated to the natural isomorphisms $(g_1g_2)\otimes X\rightarrow g_1\otimes(g_2\otimes X)$ whenever $g_i\in G$ and $X\in G/H$. Module structures are classified by non-comologous cocycles so we take as module associativity constraint any representative of the cohomology class $[\mu]$. Identifying $\mu\in H^2(G, Fun(G/H, k^{\times})) = H^2(G, Ind_H^Gk^{\times})$ with its image in $H^2(H, k^{\times})$ by Shapiro's Lemma we may classify such constraints by $H^2(H, k^{\times})$. 
The categories $Vec_G$ and $\Rep(G)$ are Morita equivalent via $Vec$: $(Vec_G)_{Vec}^*\simeq\Rep(G)$ (send representation $(V, \rho)$ to the functor $Vec\rightarrow Vec$ having $F(k) = V$ with $Vec_G$-linearity given by $\rho$). Since $\Rep(G)$ is braided the category $\Rep(G)$-Mod has monoidal structure $\boxtimes_{\Rep(G)}$. Although $Vec_G$ is not braided the  category $Vec_G$-Mod has monoidal structure as follows. For $\mc{M}, \mc{N}\in Vec_G$-Mod define new $Vec_G$-module category structure on $\mc{M}\boxtimes \mc{N}$ by $g\otimes (m\boxtimes n) := (g\otimes m)\boxtimes (g\otimes n)$ for simple object $k_g:= g$ in $Vec_G$, and linearly extend to all of $Vec_G$. Let $\mc{M}\odot\mc{N}$ denote $\mc{M}\tens{}\mc{N}$ with this module category structure.
\begin{prop}[$Vec_G$-Mod fusion rules]\label{Vec_G Frules} With notation as above 
\begin{equation*}
\mc{M}(H, \mu)\odot\mc{M}(K, \sigma)\simeq\bigoplus_{HaK\in H\setminus G/K}\mc{M}(H\cap{}^aK, \mu\sigma^{a}).
\end{equation*}
\end{prop}
\begin{proof}
Send $\langle H, \sigma\rangle$ to module category $\mc{M}(H, \sigma)$. This association is clearly well defined and respects the action of $G$. Applying the proof above for decomposition of basic elements in $\Omega(G)$ to simple objects in $\mc{M}(H, \mu)\odot\mc{N}(K, \sigma)$ verifies the stated decomposition on the level of objects. We must check only the module associativity constraints for the summand categories. To do this we simply evaluate associativity for a simple object in the summand category having set of objects $G/H\cap{}^aK$. We may choose representative $H\boxtimes aK$. For $g, h\in G$ we have
\begin{equation*}
gh\otimes (H\boxtimes aK)\simeq g\otimes(h\otimes H)\boxtimes g\otimes (h\otimes aK)
\end{equation*}
via $\mu(g, h)(H)\boxtimes\sigma(g, h)(aK)$. Noting that $G/K\simeq G/{}^aK$ as $G$-sets, restricting $\varphi: H^2(G, Fun(G/K, k^{\times}))\simeq H^2({}^aK, k^{\times})$ to coset $aK$ on the right gives $\varphi(\sigma)(k_1, k_2) = \sigma(k_1, k_2)(aK)$ for $k_1, k_2\in {}^aK$. Thus $\varphi(\sigma)^a(k_1, k_2) = \varphi(\sigma)(k_1^a, k_2^a)\in H^2({}^aK, k^{\times})$, and this we simply denote by $\sigma^a$; module associativity is $\mu\boxtimes\sigma^{a}$ which is idential to $\mu\sigma^a$ since each is a scalar on simple objects.
\end{proof}

\begin{cor}\label{InvIrrObj} The Grothendeick group of invertible irreducible $Vec_G$-module categories is isomorphic to $H^2(G, k^{\times})$.
\end{cor}
\begin{proof}
Without taking twisting into consideration, invertible irreducible $Vec_G$-module categories correspond to invertible basis elements of the Burnside ring $\Omega(G)$. Suppose $\lb H\rb\lb H'\rb = \lb G\rb$ in $\Omega(G)$. Then $\sum\lb H\cap {}^aH'\rb = \lb G\rb$ which can happen only if there is a single double coset $HH'$ and if $H\cap{}^aH' = G$, and this occurs only if $H = H' = G$. It follows from Proposition \ref{Vec_G Frules} that 
\begin{equation*}
\mc{M}(G, \mu)\odot\mc{M}(G, \mu') = \mc{M}(G, \mu\mu')
\end{equation*}
Sending $\mc{M}(G, \mu)$ to $\mu$ gives the desired isomorphism.
\end{proof}
%
%
%
%
We have an equivalence of 2-categories $Vec_G$-Mod $\rightarrow\Rep(G)$-Mod defined by sending $\mc{M}\mapsto\overline{\mc{M}}$ where
\begin{equation}
\overline{\mc{M}}:=Fun_{Vec_G}(Vec, \mc{M}).
\end{equation}
Observe that $Fun_{Vec_G}(Vec, Vec)$ acts on $Fun_{Vec_G}(\mc{M}, \mc{N})$ on the right by the formula $(F\otimes S)(M) = F(M)\odot S(k)$ whenever $M\in\mc{M}$ and $S:Vec\rightarrow Vec$ is a $Vec_G$-module functor. $F\otimes S$ is trivially a $Vec_G$-module functor: 
\begin{eqnarray*}
(F\otimes S)(g\otimes M) &\simeq& (g\otimes F(M))\odot S(k)\\
&=& (g\otimes F(M))\odot (g\otimes S(k))\\
&=& g\otimes(F(M)\otimes S(k))\\
&=& g\otimes (F\otimes S)(M).
\end{eqnarray*}
The isomorphism is $Vec_G$-linearity of $F$ and the second line follows from the fact that simple objects of $Vec_G$ (one dimensional vector spaces) act trivially on $Vec$. Let $T:Vec\rightarrow Vec$ over $Vec_G$. Associativity of the action is also trivial:  
\begin{eqnarray*}
(F\otimes ST)(M) &=& F(M)\odot ST(k)\\
&=& F(M)\odot S(k\otimes T(k))\\
&=& F(M)\odot (S(k)\otimes T(k))\\
&=& (F(M)\odot S(k))\odot T(k)\\
&=& (F\otimes S)(M)\odot T(k) = ((F\otimes S)\otimes T)(M) 
\end{eqnarray*}
The second line is tensor product (composition) in $Fun_{Vec_G}(Vec, Vec)$ and the isomorphism is due to the canonical action of $Vec$ on $\mc{N}$ given by internal hom.
\begin{prop}\label{rep vec duality} For $H<G$ and $\mu\in H^2(H, k^{\times})$ denote by $\Rep_{\mu}(H)$ the category of projective representations of $H$ with Schur multiplier $\mu$. Then $\Rep_{\mu}(H)\simeq\overline{\mc{M}(H, \mu)}$ as $\Rep(G)$-module categories.
\end{prop}
\begin{proof}Send functor $F:Vec\rightarrow\mc{M}(H, \mu)$ to $F(k)$. $\Rep(G)$-module structure on $\Rep_{\mu}(H)$ is given by $\res\otimes id$: for $V\in\Rep(G)$ and $W\in\Rep_{\mu}(H)$ the action is defined by $V\otimes W :=\res^G_H(V)\otimes W$ where $\otimes$ on the right is tensor product in $\Rep_\mu(G)$.
\end{proof}
One of the main results of this section is the following theorem.
\begin{thm}\label{prime theorem} The 2-equivalence $\mc{M}\mapsto\overline{\mc{M}}$ between $(Vec_G$-Mod$, \odot)$ and $(\Rep(G)$-Mod$, \boxtimes_{\Rep(G)})$ is monoidal in the sense that 
\begin{equation*}
\overline{\mc{M}\odot\mc{N}} \simeq \overline{\mc{M}}\boxtimes_{\Rep{G}}\overline{\mc{N}}
\end{equation*}
as $\Rep(G)$-module categories.
\end{thm}
The action of $\Rep(G)\simeq Fun_{Vec_G}(Vec, Vec)$ is given by composition of functors. Since the correspondence is an equivalence of 2-categories we may identify abelian categories of 1-cells:
\begin{equation}\label{1-cell equ}
Fun_{Vec_G}(\mc{M}, \mc{N})\simeq Fun_{\Rep(G)}(\overline{\mc{M}}, \overline{\mc{N}}).
\end{equation}
In what follows we provide a few lemmas which show that useful formulas provided earlier for monoidal 2-categories hold also over the category of $Vec_G$-modules.
\begin{lem}\label{RepG-modEq lem} The 2-equivalence $\mc{M}\mapsto\overline{\mc{M}}$ from $Vec_G$-Mod to $\Rep(G)$-Mod when restricted to 1-cells is an equivalence of right $\Rep(G)$-module categories.
\end{lem}
\begin{proof} The equivalence of 1-cells $\zeta:Fun_{Vec_G}(\mc{M}, \mc{N})\simeq Fun_{\Rep(G)}(\overline{\mc{M}}, \overline{\mc{N}})$ takes functor $F:\mc{M}\rightarrow\mc{N}$ over $Vec_G$ to the functor defined by $Q\mapsto FQ$ for $\Rep(G)$-module functor $Q:Vec\rightarrow \mc{M}$. We must check that this correspondence respects $\Rep(G)$ action. 

Any functor $E:Vec\rightarrow Vec$ over $Vec_G$ determines representation $E(k)$, and any representation $V$ determines functor $E^{V}(k) = V$. $V\in\Rep(G)\simeq\overline{Vec}$ right-acts on $F\in Fun_{\Rep(G)}(\overline{\mc{M}}, \overline{\mc{N}})$ by $(F\otimes V)(Q) = F(Q)\circ E^V$. 
Writing $<\zeta(F), Q>$ for the functor in $\overline{\mc{N}}$ determined by $F , Q$ we have, for $W\in Vec$,
\begin{eqnarray*}
<\zeta(F\otimes E^V), Q>(W) &=& (F\otimes E^V)(Q)(W)\\
&=& FQE^V(W) \\
&=& <\zeta(F)\otimes E^V, Q>(W).
\end{eqnarray*}
\end{proof}
\begin{lem}\label{RepLem1/2} Let $\mc{M}$, $\mc{N}$ be left $Vec_G$-module categories. Then 
$\mc{M}\odot\mc{N}\simeq Fun(\mc{M}^{op}, \mc{N})$ as left $Vec_G$-module categories. 
\end{lem}
\begin{proof} Let $\mc{M}:=\mc{M}(H, \mu)$ and $\mc{N}:=\mc{M}(K, \sigma)$ as above. Define 
\begin{equation}
\Phi:\mc{M}\odot\mc{N}\rightarrow Fun(\mc{M}^{op}, \mc{N}),\quad\Phi(M\odot N)(M'):=\Hom(M', M)\otimes N.
\end{equation}
Clearly $\Phi$ is an equivalence of abelian categories (see Lemma \ref{IJQInverseLemma} for example) and it remains to show that it respects $Vec_G$-module structure. The category $Fun(\mc{M}^{op}, \mc{N})$ carries $Vec_G$-module structure $(g\otimes F)(M):= g\otimes F(g^{-1}\otimes M)$ for simple objects $g$ in $Vec_G$. Left action on $\mc{M}^{op}$ is given by $X\otimes^{op}M = X\otimes M$ with inverse module associativity. We have
\begin{eqnarray*}
(gh\otimes F)(M) &=& gh\otimes F(h^{-1}g^{-1}\otimes M)\\
&\simeq& g\otimes(h\otimes F(h^{-1}\otimes(g^{-1}\otimes M)))\\
&=& g\otimes(h\otimes F)(g^{-1}\otimes M) = (g\otimes(h\otimes F))(M)
\end{eqnarray*}
where $\simeq$ is $\sigma(g, h)\mu^{-1}(h^{-1}, g^{-1})$ which is cohomologous to $\sigma(g, h)\mu(g, h)$, i.e. module associativity on functors is given by $\mu\sigma$. For simple objects $M, M'$ in $\mc{M}$, $N\in\mc{N}$
\begin{eqnarray*} (g\otimes\Phi(M\odot N))(M') &=& g\otimes(\Hom(g^{-1}\otimes M', M)\otimes N)\\
&\simeq&\Hom(M', g\otimes M)\otimes (g\otimes N)\\
&=&\Phi(g\otimes(M\odot N))(M') 
\end{eqnarray*}
where $\simeq$ is canonical. $\Phi$ respects $Vec_G$-module structure.
\end{proof}
\begin{lem}\label{RepLem 2} $Fun_{Vec_G}(\mc{M}, \mc{N})\simeq \overline{\mc{M}^{op}\odot\mc{N}}$ as right $\Rep(G)$-module categories.
\end{lem}
\begin{proof} We have an equivalence $\psi : Fun_{Vec_G}(\mc{M}, \mc{N})\rightarrow \overline{\mc{M}^{op}\odot\mc{N}}$, $F\mapsto \psi F$ where $\psi F(V)(M)  := F(M)\odot V$ whenever $V\in Vec, M\in \mc{M}$ and where we have used Lemma \ref{RepLem1/2} to express $\mc{M}^{op}\odot\mc{N}$ as category of functors is an equivalence. $\psi$ has quasi-inverse 
$F\mapsto F(k)$:
\begin{eqnarray*}
<\psi(F\otimes V), W>(M) &=& (F(M)\odot V)\odot W\\
&\simeq& F(M)\odot(V\otimes W)\\
&=& \psi F(V\otimes W)(M)\\
&=& \psi F(E^V(W))(M) = <\psi F\circ E^V,W>(M).
\end{eqnarray*}
\end{proof}
\begin{lem}\label{RepLem 1} $\overline{\mc{M}^{op}}\simeq\overline{\mc{M}}^{op}$ as $\Rep(G)$-module categories.
\end{lem}
\begin{proof} $Fun_{Vec_G}(Vec, \mc{M}^{op})\simeq Fun_{Vec_G}(\mc{M}, Vec)\simeq Fun_{\Rep(G)}(\overline{\mc{M}}, \Rep(G))$
where first $\simeq$ is Lemma \ref{RepLem1/2} and the second comes from the 2-equivalence. The first term is $\overline{\mc{M}^{op}}$ and the last is $\overline{\mc{M}}^{op}$.
\end{proof}
\begin{proof}[Proof of Theorem \ref{prime theorem}] With notation as above,
\begin{eqnarray*}
\overline{\mc{M}\odot\mc{N}}&\simeq& Fun_{Vec_G}(\mc{M}^{op}, \mc{N})\\
&\simeq& Fun_{\Rep(G)}(\overline{\mc{M}^{op}}, \overline{\mc{N}})\\
&\simeq& Fun_{\Rep(G)}(\overline{\mc{M}}^{op}, \overline{\mc{N}})\simeq \overline{\mc{M}}\boxtimes_{\Rep(G)}\overline{\mc{N}}.
\end{eqnarray*}
First line is Lemma \ref{RepLem 2}, second is Lemma \ref{RepG-modEq lem} and third is Lemma \ref{RepLem 1}.
\end{proof}
Theorem \ref{prime theorem}, together with the observation in Remark \ref{rep vec duality}, immediately gives a formula for $\Rep(G)$-module fusion rules.
\begin{cor}[$\Rep(G)$-Mod fusion rules] The twisted Burnside ring $\Omega(G)$ is isomorphic to the ring $K_0(\Rep(G)\textrm{-Mod})$ of equivalence classes of $\Rep(G)$-module categories with multiplication induced by $\boxtimes_{\Rep(G)}$. That is, for irreducible $\Rep(G)$-module categories $Rep_{\mu}(H), Rep_{\sigma}(K)$ we have, as $\Rep(G)$-module categories 
\begin{equation}\label{Rep(G)FusionFormula}
\Rep_{\mu}(H)\boxtimes_{\Rep(G)}Rep_{\sigma}(K)\simeq\bigoplus_{HaK\in H\setminus G/K} \Rep_{\mu\sigma^a}(H\cap{}^aK).
\end{equation}
\end{cor}
\begin{cor}\label{InvIrrRep} The group of invertible irreducible $\Rep(G)$-module categories is isomorphic to $H^2(G, k^{\times})$.
\end{cor}
\begin{proof} The proof is equivalent to that of Corollary \ref{InvIrrObj}.
\end{proof}
\begin{note} Corollary \ref{InvIrrRep} generalizes Corollary 3.17(ii) in \cite{ENO:homotop} where it was given for finite abelian groups. Indeed when $A$ is abelian $Vec_A = \Rep(A^*)$ for $A^*$ group homomorphisms $\Hom(A, k^{\times})$.
\end{note}
%
%
%
%
%
%
%
%
%
%
%
%
%
%


\end{document}